\documentclass[12pt]{article}
\usepackage{latexsym,amsfonts,amssymb,amsmath,amsthm,ulem}
\usepackage{graphicx}
\usepackage{epsfig,epstopdf,amssymb,color}
\usepackage[capposition=bottom]{floatrow}
\begin{document}
\setlength{\baselineskip}{16pt}

\parindent 0.5cm
\evensidemargin 0cm \oddsidemargin 0cm \topmargin 0cm \textheight 22cm \textwidth 16cm \footskip 2cm \headsep
0cm

\newtheorem{theorem}{Theorem}[section]
\newtheorem{lemma}{Lemma}[section]
\newtheorem{proposition}{Proposition}[section]
\newtheorem{definition}{Definition}[section]
\newtheorem{example}{Example}[section]
\newtheorem{corollary}{Corollary}[section]

\newtheorem{remark}{Remark}[section]
\numberwithin{equation}{section}

\def\p{\partial}
\def\I{\textit}
\def\R{\mathbb R}
\def\C{\mathbb C}
\def\u{\underline}
\def\l{\lambda}
\def\a{\alpha}
\def\O{\Omega}
\def\e{\epsilon}
\def\ls{\lambda^*}
\def\D{\displaystyle}
\def\wyx{ \frac{w(y,t)}{w(x,t)}}
\def\imp{\Rightarrow}
\def\tE{\tilde E}
\def\tX{\tilde X}
\def\tH{\tilde H}
\def\tu{\tilde u}
\def\d{\mathcal D}
\def\aa{\mathcal A}
\def\DH{\mathcal D(\tH)}
\def\bE{\bar E}
\def\bH{\bar H}
\def\M{\mathcal M}
\renewcommand{\labelenumi}{(\arabic{enumi})}

\def\disp{\displaystyle}
\def\undertex#1{$\underline{\hbox{#1}}$}
\def\card{\mathop{\hbox{card}}}
\def\sgn{\mathop{\hbox{sgn}}}
\def\exp{\mathop{\hbox{exp}}}
\def\OFP{(\Omega,{\cal F},\PP)}
\newcommand\JM{Mierczy\'nski}
\newcommand\RR{\ensuremath{\mathbb{R}}}
\newcommand\CC{\ensuremath{\mathbb{C}}}
\newcommand\QQ{\ensuremath{\mathbb{Q}}}
\newcommand\ZZ{\ensuremath{\mathbb{Z}}}
\newcommand\NN{\ensuremath{\mathbb{N}}}
\newcommand\PP{\ensuremath{\mathbb{P}}}
\newcommand\abs[1]{\ensuremath{\lvert#1\rvert}}
\newcommand\normf[1]{\ensuremath{\lVert#1\rVert_{f}}}
\newcommand\normfRb[1]{\ensuremath{\lVert#1\rVert_{f,R_b}}}
\newcommand\normfRbone[1]{\ensuremath{\lVert#1\rVert_{f, R_{b_1}}}}
\newcommand\normfRbtwo[1]{\ensuremath{\lVert#1\rVert_{f,R_{b_2}}}}
\newcommand\normtwo[1]{\ensuremath{\lVert#1\rVert_{2}}}
\newcommand\norminfty[1]{\ensuremath{\lVert#1\rVert_{\infty}}}

\newcommand{\ds}{\displaystyle}

\title{Persistence and Extinction of Nonlocal Dispersal Evolution Equations in  Moving Habitats}

\author{
Patrick De Leenheer\\
Department of Mathematics\\
Oregon State University\\
Corvallis, OR 97331\\
\\
Wenxian Shen\\
Department of Mathematics and Statistics\\
Auburn University, AL 36849\\
\\
 and\\
 \\
  Aijun Zhang\\
  Department of Mathematics\\
Oregon State University\\
Corvallis, OR 97331\\
CORRESPONDENCE AUTHOR EMAIL: zhangai@tigermail.auburn.edu.
 }

\date{}
\maketitle

\noindent {\bf Abstract.}
{This} paper is devoted to the study of persistence and extinction  of a species modeled by  nonlocal dispersal evolution equations in moving habitats with moving speed $c$. It is shown that the species  becomes extinct  if the moving speed $c$ is larger  than the so called spreading speed $c^*$, where $c^*$ is determined by the maximum linearized growth rate function. If the moving speed $c$ is smaller than $c^*$,
it is shown that the persistence of the species  depends on the patch size of the habitat, namely, the species persists if the patch size is greater than some number $L^*$ and in this case, there is a traveling wave solution with speed $c$, and it becomes extinct if the patch size is {smaller} than $L^*$.

\bigskip


\bigskip

\newpage
\section{Introduction}
\setcounter{equation}{0}
Nonlocal dispersal equations have been widely employed as models in the applied fields such as biology, material science, neuroscience, chemistry and ecology \cite{BeCoVo, CoDu, CoDaMa2, ShZh1,ShZh2,ShZh3,Schumacher}.

The current paper is to investigate  the following nonlocal dispersal equation,
\begin{equation}
\label{main-eq0}
\frac{\partial u(t,x)}{\partial t}=\int_{\RR} k(y-x) u(t,y)dy-u(t,x)+f(x-ct,u)u(t,x),\quad x\in\RR.
\end{equation}
We assume that $k(\cdot): \RR \to \RR^+$ is a $C^1$ convolution kernel function {that} satisfies the following:

\medskip
\noindent{\bf (H1)}  {\it $k(\cdot)\in C^1(\RR,[0,\infty))$, $k(z)=k(-z)$,
 $\int_{\RR}k(z)dz=1$, $k(0)>0$, and there exist $\mu,M>0$ such that $k(z)< e^{-\mu|z|}$ and
 $|k'(z)|< e^{-\mu|z|}$ for $|z|>M$.}

 \medskip
  Typical examples satisfying (H1) include the probability density function of the normal distribution $k(x)=\frac{1}{\sqrt{2\pi}}e^{-\frac{x^2}{2}}$ and any $C^1$ {symmetric} convolution kernel functions supported on a bounded interval. Biologically, in \eqref{main-eq0},  the term
 $\int_{\RR} k(y-x)u(t,y)dy-u(t,x)$ characterizes the dispersal of the organisms  that exhibits long range internal interactions. $f(x-c t,u)$ is the reaction term that is related to the growth of species. Noting a speed $c$ in the reaction term $f(x-ct,u)$, biologically we assume the reaction of the populations will change with the moving habitat of speed $c$ due to some external environment change, like climate change. {Without loss of generality, we assume that $c\ge 0$. If $c<0$, biologically it means that the habitat moves in an opposite direction. Mathematically, by changing variables with $\tilde c=-c$ and $\tilde x=-x$, we can obtain an equivalent equation as \eqref{main-eq0} for $\tilde c>0$ and $\tilde x$.}
 { Let $\phi_\pm (x)$ be  $C^1$ functions satisfying that $\phi_\pm(\pm x)=1$  for $x\le 0$, $\phi_\pm(\pm x)=0$ for $x\ge 1$, {$\phi_+(x)=\phi_-(-x)$,} $\phi^{'}_+(x)\le 0$ and $\phi_-^{'}(x)\ge 0$ for $x\in\RR$.}
 We assume that $f$ satisfies

 \medskip

\noindent{\bf (H2)}  {\it There are $r,q,L,L_0>0$ such that $f(x,u)$ is $C^1$ in {$(x,u)$}; $f(x,u)=-q$ for $|x|\ge L+L_0$;
$f(x,u)=r(1-u)$ for $|x|\le L$; { $f(x,u)=-q+(r(1-u)+q)\phi_+\big( \frac{x-L}{L_0}\big)$ for $L<x<L+L_0$; and $f(x,u)=-q+(r(1-u)+q)\phi_-\big(\frac{x+L}{L_0}\big)$ for $-L-L_0<x<-L$}.
}
\medskip

Observe that { $f(x,\cdot)=f(-x,\cdot)$ and}
$$\lim_{L_0\to 0^+} f(x,u)=\begin{cases}-q \quad {\rm for}\quad x > L\cr
r(1-u) \quad {\rm for}\quad -L\le x \le L\cr
-q\quad {\rm for}\quad x<-L.
\end{cases}
$$
Here is an example of $f(x,u)$ which satisfies (H2),
$$
f(x,u)=\begin{cases} -q \quad {\rm for}\quad x\ge L_0+L\cr
-q+\frac{r(1-u)+q}{2}\big(1+\cos \frac{\pi (x-L)}{L_0}\big)\quad {\rm for}\quad L< x<L_0+L\cr
r(1-u) \quad {\rm for}\quad -L\le x \le L\cr
-q+\frac{r(1-u)+q}{2}\big(1+\sin \frac{\pi(2x+2L+L_0)}{2L_0} \big)\quad {\rm for}\quad -L_0-L< x<-L\cr
-q\quad {\rm for}\quad x \le -L_0-L.
\end{cases}
$$

Assumption (H2) indicates that the region $|x|\le L$ is the favorable habitat for the species;  there is a finite mortality rate $q$
 outside of the region $|x|\le L+L_0$;  and  the region $L\le |x|\le L+L_0$ is the transition region.

 {Recently, modeling the effects of global climate change on populations has drawn a lot research attention in the scientific community} \cite{BeDiNaZe, LiBeShFa,ZhKo}. Berestycki {et al.} in \cite{BeDiNaZe} and Li {et al.} in  \cite{LiBeShFa} considered the following reaction-diffusion equation, for $c>0$
 \begin{equation}
\label{main-RD-eq0}
\frac{\partial u(t,x)}{\partial t}=D \frac{\partial^2 u(t,x)}{\partial x^2} +f(x-c t,u)u(t,x),\quad x\in\RR.
\end{equation}
but with different reaction term $f$. In \cite{BeDiNaZe}, {$f$ is assumed to be}
 \begin{equation}
\label{Reaction-Berestycki}
f(x-ct,u)=\begin{cases}
r-u\quad {\rm for}\quad  |x-ct| \leq L \cr
-q  \quad {\rm otherwise}
\end{cases}
\end{equation} 
for some $L>0$, {while in \cite{LiBeShFa}, $f$ is of the form}
 \begin{equation}
\label{Reaction-Li}
f(x-ct,u)=r(x-ct)-u,
\end{equation}
where $r$ is continuous, non-decreasing and bounded with $r(-\infty)<0$ and $r(\infty)>0$.
Other researchers {have considered discrete dynamical systems}, including difference equations \cite{ZhKo} and lattice differential equations  \cite{HuLi}. There is a recent work \cite{LiWangZhao} on \eqref{main-RD-eq0} with the random dispersal {term} being replaced by nonlocal dispersal {term} and with the nonlinear term  \eqref{Reaction-Li}, where the authors studied the spatial dynamics.

Interesting dynamical issues for \eqref{main-eq0} and \eqref{main-RD-eq0} include the persistence and extinction {of the population}, in particular, the dependence of the persistence and extinction on the speed $c$ and the patch size of the moving habitat.
It will be seen that the persistence and existence of traveling wave solutions with speed $c$ are strongly related, and the extinction and nonexistence of traveling wave solutions with speed $c$ are {also} strongly related.
In the current paper, we are interested in the existence and nonexistence of traveling wave solutions of \eqref{main-eq0} with speed $c$, i.e., {positive} solutions of the form
$u(t,x)=v(x-ct)$.

To this end, we consider solutions of \eqref{main-eq0} of the form  $u(t,x)=v(t,x-ct)$ with $v(t,x)$ being differentiable. Then letting $\xi=x-ct$, $v(t,\xi)$ satisfies
\begin{equation}
\label{main-eq}
\frac{\partial v(t,\xi)}{\partial t}=c \frac{\partial v(t,\xi)}{\partial \xi}+ \int_{\RR} \kappa(\eta-\xi) v(t,\eta)d\eta-v(t,\xi)+f(\xi,v) v(t,\xi),\quad \xi \in\RR.
\end{equation}

We remark that {equation \eqref{main-eq} models the nonlocal dispersal, advection and reaction of a single species in a heterogeneous environment.} The number $c$ measures the advection velocity. The term {$c \frac{\partial v(t,\xi)}{\partial \xi}$} describes the drift of the population with the constant speed $c$. {Advective processes occur, for example, in a river or ocean,} where organisms may drift, sink or rise due to the water flows and their own relative weights compared with the surrounding medium (i.e water).

Note that {any nontrivial stationary solution $v(\xi)$ of \eqref{main-eq} satisfies
\begin{equation}
\label{main-eq1}
c v'(\xi)+\int_{\RR} \kappa(\eta-\xi) v(\eta)d\eta-v(\xi)+f(\xi,v) v(\xi)=0,\quad \xi \in\RR,
\end{equation}
and gives rise to a traveling wave solution} $u(t,x)=v(x-ct)$ of \eqref{main-eq0}.

Let
$$
X=C_{\rm unif}^b(\RR)=\{u\in C(\RR)\,|\, u\quad \text{is uniformly continuous and bounded on}\quad \RR\}
$$ with norm $\|u\|=\ds\sup_{x\in\RR}|u(x)|$,
and
$$
X^+=\{u\in X\,|\, u(x)\ge 0\}.
$$
Consider initial value problem for \eqref{main-eq0} and \eqref{main-eq} on $X$. By semigroup theory (See \cite{Henry,Paz}), for any $u_0\in X$,  \eqref{main-eq0} has a unique local classical solution $u(t,x;u_0)$ with $u(0,x;u_0)=u_0(x)$, and
for any $u_0\in X$, and \eqref{main-eq} has a unique local mild solution $v(t,\xi;u_0)$ with $v(0,\xi;u_0)=u_0(\xi)$. Moreover, if $u_0$ is
differentiable and $u_0^{'}(\cdot)\in X$, then $v(t,\xi;u_0)$ is the classical solution of \eqref{main-eq}.

We say that {\it persistence} occurs in \eqref{main-eq0} if { for any $u_0\in X^+$ with $\ds\inf_{x\in\RR}u_0(\xi)>0$},
$$
\liminf_{t\to\infty} \inf_{|\xi|\le K} v(t,\xi;u_0)>0
$$
for any $K>0$. We say that {\it extinction} occurs in \eqref{main-eq0} if for any $u_0\in X^+$,
{$$
\lim_{t\to\infty} \sup_{\xi \in \RR}v(t,\xi;u_0)=0.
$$}

 For $r>0$, let $c^*$ be the spreading speed of
 \begin{equation}
 \label{main-fixed-domain-eq}
 u_t=\int_{\RR}k(y-x)u(t,y)dy-u(t,x)+r(1-u)u, \quad x\in\RR,
 \end{equation}
 that is,
 \begin{equation}
 \label{c-star-eq}
 c^*=\inf_{\mu>0}\frac{\int_{\RR} e^{-\mu z}k(z)dz-1+r}{\mu}
 \end{equation}
 (see Proposition \ref{c-star-prop}).

For given $\lambda>-q$, let
$$
g(\mu;\lambda)=c\mu+\int_{\RR} e^{\mu\eta}k(\eta)d\eta -1-q-\lambda.
$$
Then
$$
g_{\mu \mu}= \int_{\RR}\eta^2e^{\mu\eta}k(\eta)d\eta>0,
$$
and $g(0;\lambda)=-q-\lambda<0$. Note that $g(\mu;\lambda)\to \infty$ as $\mu\to\pm\infty$.  Hence there are $\mu_-(\lambda)<0<\mu_+(\lambda)$ such that
\begin{equation}
\label{mu-plus-minus-eq}
g(\mu_\pm(\lambda);\lambda)=0.
\end{equation}

The main results of the current paper can {then} be stated as follows.

 \medskip

 \noindent $\bullet$ (Tail properties of traveling wave solutions) {\it
Suppose that $v(\xi)=\Phi(\xi)$ is a bounded positive solution of \eqref{main-eq1}. Then  
$$
\limsup_{\xi\to \pm \infty}\frac{\Phi(\xi)}{e^{\mu_\mp\xi}}<\infty,
$$
where $\mu_\mp=\mu_\mp(0)$}
(see Theorem \ref{tail-thm}).

\medskip

 \noindent $\bullet$ (Equivalence of persistence and existence of traveling wave solutions) {\it The following two statements are equivalent:
 persistence occurs in \eqref{main-eq0} and there are traveling wave solutions of \eqref{main-eq0} with speed $c$, which implies that the following two statements are equivalent: extinction occurs in \eqref{main-eq0} and there are no traveling wave solutions of \eqref{main-eq0} with speed $c$ } (see Theorem \ref{equivalence-thm} and Remark \ref{equivalence-rk}).

\medskip

 \noindent $\bullet$   (Existence{, uniqueness, and }{nonexistence}  of traveling wave solution)
 {\it There is $0\le L^*\le \infty$ such that {if $L^*<L<\infty$, then \eqref{main-eq} has a {unique} positive stationary solution
$v(t,\xi)=\Phi(\xi)$. On the other hand, if $0<L<L^*$, then there is no positive stationary solution of \eqref{main-eq}.}
Moreover, if $0\le c<c^*$, then $0\le L^*<\infty$,  and if $c\ge c^*$, then $L^*=\infty$}  (see Theorems \ref{existence-thm} {and \ref{uniqueness-thm}}).

\medskip

\medskip

Observe that
persistence and  extinction  in \eqref{main-eq} are strongly  related to the spectral problem of
the linearization of \eqref{main-eq} at the trivial solution $v\equiv 0$,
\begin{equation}
\label{main-lin-eq1}
\frac{\partial v(t,\xi)}{\partial t}=c\frac{\partial v(t,\xi)}{\partial \xi}+ \int_{\RR} \kappa(\eta-\xi) v(t,\eta)d\eta-v(t,\xi)+f(\xi,0) v(t,\xi),\quad \xi \in\RR.
\end{equation}
Let $\lambda(c,L)$ be the {principal spectral point (see Definition \ref{principal-spectrum-def})} of the spectral problem associated with \eqref{main-lin-eq1}.
We also prove that

\medskip

\noindent $\bullet$ (Principal eigenvalue)  {\it $\lambda(c,L)$ is a principal eigenvalue {(see Definition \ref{principal-spectrum-def})}. Moreover, {if $0<  c< c^*$, then} there is $0\le L^{**}\le \infty$ such that $\lambda(c,L) >0$ for all $L> L^{**}$, and for any $0<L<L^{**}$, $\lambda(c,L) \le 0$. {If $c>c^*$, then} $\lambda(c,L) < 0$ for all $L>0$}
(see Theorem \ref{new-aux-main-thm}).

\medskip

\noindent $\bullet$ (Persistence/extinction)
{\it  If $\lambda(c,L)>0$, then there is a positive stationary solution of \eqref{main-eq}, and for any $K>0$
and $u_0\in X^{++}:=\{u\in X|u>0\}$ satisfying $\ds\liminf_{\xi\to\infty}\frac{u_0(\xi)}{e^{\mu_-(\lambda(c,L))\xi}}>0$ and $\ds\liminf_{\xi\to -\infty}\frac{u_0(\xi)}{e^{\mu_+(\lambda(c,L))\xi}}>0$,
$$
\liminf_{t\to\infty} \inf_{|\xi|\le K} v(t,\xi;u_0)>0,
$$
where $v(t,\xi;u_0)$ is the solution of \eqref{main-eq} with $v(0,\xi;u_0)=u_0(\xi)$, and $\mu_\pm (\lambda(c,L))$ are as in \eqref{mu-plus-minus-eq}.
 If {$\lambda(c,L)\le 0$}, then for any $u_0\in X^+$,
$$
\lim_{t\to\infty} \sup_{x\in\RR} v(t,\xi;u_0)=0.
$$ }
(see Theorem \ref{persistence-extinction}).

 \medskip

Note that, by the equivalence of the occurrence of persistence and the existence of traveling wave solutions, {$L^{**}= L^*$}.


The rest of the paper is organized as follows. In section 2, we present some preliminaries such as comparison principles for nonlocal evolution equations. In section 3, we show the tail behaviors of the traveling waves.   We examine in section 4 the equivalence of the occurrence of persistence and the existence of traveling wave solutions. In section 5, we prove the existence, {uniqueness} and nonexistence of the traveling wave solutions. In sections 6, we investigate the spectral theory of nonlocal operators and discuss {their applications to species persistence and extinction}.

\section{Preliminaries}

In this section, we present some preliminary materials to be used in the following sections.

\subsection{Comparison principle of nonlocal evolution equations}
A continuous function $v(t,\xi)$ on $[0,T)\times \RR$ is called a {\it super-solution} or {\it sub-solution} of
\eqref{main-eq} if $\frac{\p v}{\p t},\frac{\p v}{\p \xi}$ exist and are continuous on $[0,T)\times\RR$ and satisfy
$$
\frac{\p v}{\p t} \geq c \frac{\p v}{\p \xi}+ \int_{\RR} k(\eta-\xi)v(t,\eta)d\eta-v(t,\xi)+f(\xi,v)v(t,\xi),\quad \xi\in\RR
$$
or
$$
\frac{\p v}{\p t}  \leq c \frac{\p v}{\p \xi}+ \int_{\RR} k(\eta-\xi)v(t,\eta)d\eta-v(t,\xi)+f(\xi,v)v(t,\xi),\quad \xi\in\RR
$$
for $t\in [0,T)$,  respectively.
The super-/sub-solutions for the  linear equation \eqref{main-lin-eq1} are defined similarly.
\begin{proposition}[Comparison principle]
\label{comparison-linear-prop} $\quad$
\begin{itemize}
\item[(1)]
If $\underline{v}(t,\xi)$ and $\overline{v}(t,\xi)$ are sub-solution and super-solution of \eqref{main-lin-eq1} on $[0,T)$,
respectively, $\underline{v}(0,\cdot)\leq \overline{v}(0,\cdot)$,  and $\overline{v}(t,\xi)-\underline{v}(t,\xi)\geq -\beta_0$ for $(t,\xi)\in [0,T)\times
\RR$ and some $\beta_0>0$, then
$\underline{v}(t,\cdot)\leq \overline{v}(t,\cdot)\quad {\rm for}\quad t\in [0,T).$

\item[(2)] Suppose that $v_1,v_2\in X$ and $v_1\leq v_2$, $v_1\not =v_2$.
Then  $v(t,\xi;v_1)<v(t,\xi;v_2)$ for all $t>0,\xi \in \RR$, {where $v(t,\xi;v_k)$  is the solution of \eqref{main-eq} with $v(0,\xi;v_k)=v_k$ for $k=1,2$.}
\end{itemize}
\end{proposition}

\begin{proof}
(1) This follows by modifying the arguments in \cite[Proposition 2.1]{ShZh1}.
Let $v(t,\xi)=e^{\sigma t}(\overline{v}(t,\xi)-\underline{v}(t,\xi))$. Then $v(t,\xi)\geq -e^{\sigma t}\beta_0$, and
\begin{equation}
\label{ineq-aux1}
\frac{\p v}{\p t}\geq c v_{\xi}(t,\xi)+\int_{\RR} k(\eta-\xi)v(t,\eta)d\eta+p(\xi)v(t,\xi),\quad \xi\in \RR,
\end{equation}
for $t \in (0,T)$ and $p(\xi)=f(\xi,0)-1+\sigma$. Choose $\sigma>0$ such that $p(\xi)>0$ for all $(t,\xi)\in [0,T)\times \RR$. We claim that $v(t,\xi)\geq 0$ for $(t,\xi)\in [0,T)\times \RR$.

Let $p_0=\ds\sup_{\xi \in \RR}p(\xi)$. Let $T_0=\min\{T,\frac{1}{p_0+1}\}$. Let $\xi=\hat{x}-c t$ and $\eta=\hat{y}-c t$, then
\begin{equation}
\label{ineq-aux2}
\frac{\p v(t,\hat{x}-c t)}{\p t}\geq \int_{\RR} k(\hat{y}-\hat{x})v(t,\hat{y}-c t)d\hat{y}+p(\hat{x}-c t)v(t,\hat{x}-c t),\quad \hat{x}\in \RR,
\end{equation}
for $t \in (0,T)$.

Assume that there are $\tilde t\in (0,T_0)$ and $\tilde x\in\RR$ such that
$v(\tilde t,\tilde x)<0$. Let
$$
v_{\inf}:=\inf_{(t,\hat{x})\in [0,\tilde t]\times\RR}v(t,\hat{x}-c t)<0.
$$
Observe that there are $t_n\in (0,\tilde t]$ and $x_n\in\RR$ such that
$$
v(t_n,x_n-c t_n)\to v_{\inf}\quad {\rm as}\quad n\to\infty.
$$
By \eqref{ineq-aux2}, we have that
\begin{align*}
v(t_n,x_n-c t_n)-v(0,x_n)&\geq \int_0^{t_n} \big[\int_{\RR} k(\hat{y}-x_n)v(t,\hat{y}-c t)d\hat{y}+p(x_n-c t)v(t,x_n-c t)
\big]dt\\
&\geq \int_0^{t_n}\big[ \int_{\RR}k(\hat{y}-x_n)v_{\inf}d\hat{y}+p_0 v_{\inf}\big]dt\\
&  =t_n(1+p_0)v_{\inf}\\
&\geq \tilde t(1+p_0)v_{\inf}
\end{align*}
for $n=1,2,\cdots$. Note that $v(0,x_n)\geq 0$ for $n=1,2,\cdots$. We then have that
$$
v(t_n,x_n-c t_n)\geq \tilde t(1+p_0) v_{\inf}
$$
for $n=1,2,\cdots$. Letting $n\to\infty$, since $\tilde t(1+p_0)< T_0(1+p_0)\leq \frac{1}{1+p_0}(1+p_0)=1$ and $v_{\inf}<0$, we get
$$
v_{\inf}\geq \tilde t(1+p_0)v_{\inf}>v_{\inf}.
$$
This is a contradiction, {that implies that $v(t,\xi)\geq 0$ for $(t,\xi)\in [0,T_0)\times \RR$. The above procedure can be repeated for $t \in [kT_0,(k+1)T_0) \bigcap [0,T)$ for $k=1,2...$.} Hence $v(t,\xi)\geq 0$ for $(t,\xi)\in [0,T)\times \RR$ and then $\underline{u}(t,\xi)\leq \overline{v}(t,\xi)$
for $(t,\xi)\in[0,T)\times \RR$.

(2) This follows from similar arguments as in \cite[Proposition 2.2]{ShZh1}.
\end{proof}

\subsection{Comparison principle for nonlocal Dirichlet boundary problems}
In this subsection, we consider the following linear nonlocal equations with non-homogeneous Dirichlet boundary conditions:
\begin{equation}
\label{Nolocal-Deq}
\begin{cases}
\frac{\partial\varphi(t,\xi)}{\partial t}=c \frac{\partial \varphi(t,\xi)}{\partial \xi}+\int_{\RR} k(\eta-\xi) \varphi(t,\eta) d\eta-\varphi(t,\xi) +q(\xi)\varphi(t,\xi),\,\, \xi \in (a,b)  \cr
\varphi(t,\xi)=g(\xi),\,\,  \xi \notin (a,b)
\end{cases}
\end{equation} for $a,b \in \RR$,  $b>a$, $q(\xi)$, $g(\xi)\in X$.

\smallskip

Let $\mathcal{L}_D v:=c v'(\xi)+{ \int_{\RR}k(\eta-\xi)v(\eta)d\eta}-v(\xi)+q(\xi)v(\xi)$ for $v,v' \in X$.
 A function $v(t,\xi)$ is called a {\it super-solution} or {\it sub-solution} of
\eqref{Nolocal-Deq} for $t\in [t_0,t_0+T]$ if $v(t,\cdot),v_\xi(t,\cdot) \in X $ and $v$ satisfies
$$
\begin{cases}
v_t(t,\xi)\ge \mathcal{L}_D v ,\quad \xi \in (a,b) \cr
{ v(t,\xi)}\geq g(\xi),\,\,  \xi \notin (a,b)
\end{cases}
$$
or
$$
\begin{cases}
v_t(t,\xi)\le \mathcal{L}_D v ,\quad \xi \in (a,b) \cr
v(t,\xi)\leq g(\xi), \,\, \xi \notin (a,b),
\end{cases}
$$
respectively, for $t\in [t_0,t_0+T]$.

\begin{proposition}
\label{comparison_Nolocal_D-prop}
 Let $\overline{u}$ and $\underline{v}$ be the super-solution and sub-solution of \eqref{Nolocal-Deq} with boundary conditions $\overline{g},\underline{g}$ respectively. Suppose that
  $\bar u(t_0,\xi)\ge \underbar v(t_0,\xi)$ for $\xi\in [a,b]$, $\overline{g}(\xi)\geq \underline{g}(\xi)$ for $\xi \notin (a,b)$ and $q(\xi) < 0$ for $\xi \in (a,b)$. Then  $\overline{u} (t,\xi)\geq \underline{v}(t,\xi)$ for $t\in[t_0,t_0+T]$ and $\xi\in [a,b]$.
\end{proposition}

\begin{proof}
Let $w=\overline{u}-\underline{v}$. Then
$$\mathcal{L}_D w=\mathcal{L}_D \overline{u}-\mathcal{L}_D \underline{v} \leq w_t.$$
We claim that $w \geq 0$. Suppose not, then there are $t^0\in[t_0,t_0+T]$,  $\xi_0 \in [a,b]$ such that $w(t^0,\xi_0)=\ds\min_{t\in[t_0,t_0+T],\xi \in [a,b]} \{w(t,\xi)\}<0.$ By $\overline{g}(\xi)\geq \underline{g}(\xi)$ for $\xi \notin (a,b)$, $w(t_0,\xi)\geq 0$ for $\xi\in [a,b]$, we have that $t^0\in (t_0,t_0+T]$ and $\xi_0 \in (a,b)$,  and then $w_t(t^0,\xi_0)\le 0$ and $w_\xi(t^0,\xi_0)=0$. Therefore, {since} $w(t^0,\eta)-w(t^0,\xi_0)\geq 0$ for $\eta \in (a,b)$, $w(t^0,\xi_0)<0$ and $q(\xi_0)< 0$, there holds 
$$
\int_{a}^{b}k(\xi_0-\eta)[w(t^0,\eta)-w(t^0,\xi_0)]d\eta{ +}\int_{\RR \setminus (a,b)}k(\xi_0-\eta)d\eta w(t^0,\xi_0)+q(\xi_0)w(t^0,\xi_0){>} 0.$$
  {Thus we have that $\mathcal{L}_D w(t^0,\xi_0)>0$, which is a contradiction.}
\end{proof}

\subsection{Convergence on compact sets}

{ In this subsection, we explore the convergence property of solutions of \eqref{main-lin-eq1} and \eqref{main-eq} in compact open topology.
Note that  $f(\xi,v)$ depends on $r$, $q$, $L$, and $L_0$. View $L$ as a parameter and  write $f(\xi,v)$ as $f(\xi,v;L)$. Let $f(\xi,v;\infty)=r(1-v)$. For fixed $r$, $q$, and $L_0$, to indicate the dependence of solutions of \eqref{main-lin-eq1}  and \eqref{main-eq} on $L$,
we denote the solution of \eqref{main-lin-eq1} (or \eqref{main-eq}) by $\tilde{v}(t,\xi;u_0,L)$ (or $v(t,\xi;u_0,L)$).

\begin{proposition}[Convergence on compact subsets]
\label{basic-convergence}
Suppose that $u_{0n},u_0\in X^+$ ($n=1,2,\cdots$) and $\{\|u_{0n}\|\}$ is bounded.
\begin{itemize}
\item[(1)] If as $n\to\infty$, $u_{0n}(\xi)\to u_0(\xi)$ uniformly in $\xi$ on bounded sets and $L_n\to \infty$,
 then for each $t>0$, $\tilde{v}(t,\xi;u_{0n},L_n)\to \tilde{v}_\infty(t,\xi;u_0)$ as $n\to\infty$ uniformly in $\xi$ on bounded sets,
 where {$\tilde{v}(t,\xi;u_{0n},L_n)$ is the solution of \eqref{main-lin-eq1} with $\tilde{v}(0,\xi;u_{0n},L_n)=u_{0n}$ and} $\tilde{v}_\infty(t,\xi;u_0)$ is the solution of
 \begin{equation}
\label{fixed-domain-eq1}
v_t=cv_\xi+\int_{\RR}\kappa(\eta-\xi)v(t,\eta)d\eta-v(t,\xi)+r v(t,\xi),\quad \xi\in\RR
\end{equation}
with $\tilde{v}_\infty(0,\xi;u_0)=u_0(\xi)$ for $\xi\in\RR$.
\item[(2)] Fix $L$. If $u_{0n}(\xi)\to u_0(\xi)$ as $n\to\infty$ uniformly in $\xi$ on bounded sets,
 then for each $t>0$, $v(t,\xi;u_{0n})\to v(t,\xi;u_0)$ as $n\to\infty$ uniformly in $\xi$ on bounded sets, where $v(t,\xi;u_{0n})$ and $v(t,\xi;u_0)$ are the solutions of \eqref{main-eq} with $v(0,\xi;u_{0n})=u_{0n}$ and $v(0,\xi;u_0)=u_0$ respectively.
 \end{itemize}
 \end{proposition}

 \begin{proof}
 It can be proved by the similar arguments as those in \cite[Proposition 3.3]{KoSh}. For completeness,
 we provide a proof in the following.
\smallskip

(1) Let $v^n(t,\xi)=\tilde{v}(t,\xi;u_{0n},L_n)-\tilde{v}_\infty(t,\xi;u_0)$.
Then $v^n(t,\xi)$ satisfies
\begin{equation*}
v^n_t(t,\xi)=cv^n_\xi(t,\xi)+\int_{\RR}\kappa(\eta-\xi)v^n(t,\eta)d\eta-v^n(t,\xi)+a_n(\xi)v^n(t,\xi)+b_n(t,\xi),
\end{equation*}
where $a_n(\xi)=f(\xi,0;L_n)$
and $b_n(t,\xi)= \tilde{v}_\infty(t,\xi;u_0)\big(a_n(\xi)-r\big).$

Note that $\{a_n(\xi)\}$ is uniformly bounded and continuous  in $\xi$ with $|a_n(\xi)| \le \max\{r,q\}$.
By (H2),
$b_n(t,\xi)\to 0$ as $n\to\infty$ uniformly in $(t,\xi)$ on bounded sets of $[0,\infty)\times\RR$.

Take a $\rho>0$. Let
$$
X(\rho)=\{u\in C(\RR,\RR)\,|\, u(\cdot) e^{-\rho \|\cdot\|}\in X\}
$$
with norm $\|u\|_\rho=\|u(\cdot)e^{-\rho \|\cdot\|}\|$.
Let
$$
\mathcal{A}u=cu'+\int_{\RR} k(\eta-\xi)u(\eta)d\eta-u(\xi)
$$
for $u\in X(\rho)$ with $u'(\cdot)\in X(\rho)$.
{ Let $\mathcal{D}:X^1(\rho)\to X(\rho)$ be defined by $\mathcal{D} u= u'$ for $u \in X^1(\rho):=\{u \in X(\rho) | u'(\cdot)\in X(\rho)\}$. Since
the resolvent, denoted by $R(\lambda,\mathcal{D})$, can be computed explicitly as
$$
R(\lambda,\mathcal{D}) u=\int_{-\infty}^{\xi}e^{-\lambda(\xi-s)}u(s)ds, \xi\in \RR.
$$
Moreover, $\|R(\lambda,\mathcal{D})\|_{X(\rho)} \leq \frac{1}{\lambda+\rho}$ for all $\lambda>0$. Thus, $\mathcal{D}$ is dissipative and $\lambda I- \mathcal{D}$ is surjective for  $\lambda>0$,where $I$ denotes the identity operator.
Hence, by Lumer–Phillips theorem (See Section 1.4 of \cite{Paz})   $\mathcal{D}$ generates a $C_0$-semigroup on $X(\rho)$. Therefore, by perturbations of bounded linear operators (Theorem 1.1 in Section 3.1 of \cite{Paz}}), $\mathcal{A}$ generates a $C_0$-semigroup on $X(\rho)$ denoted by
$e^{\mathcal{A}t}$,
and there are $M>0$ and $\omega>0$ such that
$$
\|e^{\mathcal{A}t}\|_{X(\rho)}\leq M e^{\omega t}\quad \forall t\geq 0.
$$
Hence
\begin{align*}
v^n(t,\cdot)=e^{\mathcal{A}t}v^n(0,\cdot)+\int_0^t e^{\mathcal{A}(t-\tau)}a_n(\cdot)v^n(\tau,\cdot)
d\tau+\int_0^t e^{\mathcal{A}(t-\tau)}b_n(\tau,\cdot)d\tau
\end{align*}
and
then
\begin{align*}
\|v^n(t,\cdot)\|_{X(\rho)}&\leq M e^{\omega t}\|v^n(0,\cdot)\|_{X(\rho)}+
M\sup_{\xi\in\RR} |a_n(\xi)|\int_0^ t e^{\omega(t-\tau)}\|v^n(\tau,\cdot)\|_{X(\rho)}d\tau\\
&\quad +M\int_0^ t e^{\omega(t-\tau)}\|b_n(\tau,\cdot)\|_{X(\rho)}d\tau\\
&\leq M e^{\omega t}\|v^n(0,\cdot)\|_{X(\rho)}+
M\sup_{\xi\in\RR} |a_n(\xi)|\int_0^ t e^{\omega(t-\tau)}\|v^n(\tau,\cdot)\|_{X(\rho)}d\tau\\
&\quad + \frac{M}{\omega}\sup_{\tau\in
[0,t]}\|b_n(\tau,\cdot)\|_{X(\rho)} e^{\omega t}.
\end{align*}
By Gronwall's inequality,
$$
\|v^n(t,\cdot)\|_{X(\rho)}\leq e^{\big(\omega+M\ds\sup_{\xi\in\RR} |a_n(\xi)|\big)t}\Big(M\|v^n(0,\cdot)\|_{X(\rho)}
+ \frac{M}{\omega}\sup_{\tau\in [0,t]}\|b_n(\tau,\cdot)\|_{X(\rho)}\Big).
$$
Note that $\|v^n(0,\cdot)\|_{X(\rho)}\to 0$ and $\ds\sup_{\tau\in
[0,t]}\|b_n(\tau,\cdot)\|_{X(\rho)}\to 0$ as $n\to\infty$. It then
follows that
$$
\|v^n(t,\cdot)\|_{X(\rho)}\to 0\quad {\rm as}\quad n\to\infty
$$ and then
$$
\tilde{v}(t,\xi;u_{0n},L_n)\to \tilde{v}_\infty(t,\xi;u_0)\quad {\rm as}\quad
n\to\infty
$$
uniformly in $\xi$ on bounded sets.

\smallskip

(2)  Let $v^n(t,\xi)=v(t,\xi;u_{0n})-v(t,\xi;u_0)$.
Then $v^n(t,\xi)$ satisfies
\begin{equation*}
v^n_t(t,\xi)=cv^n_\xi(t,\xi)+\int_{\RR}\kappa(\eta-\xi)v^n(t,\eta)d\eta-v^n(t,\xi)+a_n(t,\xi)v^n(t,\xi)+b_n(t,\xi),
\end{equation*}
where
$$
a_n(t,\xi)=f(\xi,v(t,\xi;u_{0n}))+v(t,\xi;u_0)\cdot \int_0^1 \p_u f(\xi,s v(t,\xi;u_{0n})+
(1-s)v(t,\xi;u_0))ds
$$
and
$$
b_n(t,\xi)= v(t,\xi;u_0)
\cdot \big(f(\xi, v_\infty(t,\xi;u_0))-f(\xi,
 v(t,\xi;u_0))\big).
$$
By the boundedness of $\{\|u_{0n}\|\}$,  $\{a_n(t,\xi)\}$ is uniformly bounded and continuous  in $t$ and $\xi$.
By (H2),
$b_n(t,\xi)\to 0$ as $n\to\infty$ uniformly in $(t,\xi)$ on bounded sets of $[0,\infty)\times\RR$.
The rest follows from a simple modification of the proof in (1).\end{proof}
}

\subsection{Spreading speeds in fixed habitats}

In this subsection, we review some properties {about} the spreading speeds of the nonlocal dispersal equation \eqref{main-fixed-domain-eq}.
{Let $\mu^*>0$ and $c^*$ satisfy that
\begin{equation}
\label{Spreading_c}
c^*=\inf_{\mu>0}\frac{\int_{\RR} e^{-\mu z}k(z)dz-1+r}{\mu}=\frac{\int_{\RR} e^{-\mu^* z}k(z)dz-1+r}{\mu^*}.
\end{equation}}

\begin{proposition}
\label{c-star-prop}
\begin{itemize}
\item[(1)] For any $u_0\in X^+$ with nonempty compact support,
{$$
\limsup_{ t\to\infty}\sup_{|x|\ge ct} u(t,x;u_0)=0\quad \forall\,\, c>c^*
$$
and
$$
\liminf_{t\to\infty}\inf_{|x|\le ct}(u(t,x;u_0)-1)=0\quad \forall\,\, 0<c<c^*,
$$}
where $u(t,x;u_0)$ is the solution of \eqref{main-fixed-domain-eq} with $u(0,x;u_0)=u_0(x)$.

\item[(2)] For any $c\ge c^*$, \eqref{main-fixed-domain-eq} has a positive traveling wave solution $u(t,x)=\phi(x-ct)$ with
$\phi(-\infty)=1$ and $\phi(\infty)=0$. {Moreover, for $\mu \in (0,\mu^*)$ such that $c=\frac{\int_{\RR} e^{-\mu z}k(z)dz-1+r}{\mu}$, $\ds\lim_{x \to \infty}\frac{\phi(x)}{e^{-\mu x}}=1$}.
\end{itemize}
\end{proposition}
\begin{proof}
(1) See Theorem E of \cite{ShZh1}.

(2) See Theorem 2.4 of \cite{ShZh2}.
\end{proof}

\section{Tail behavior of traveling waves}

In this section, {we study the tail behavior or decay behavior of  positive stationary solutions of  \eqref{main-eq1} (assuming they exist), or equivalently,  traveling wave solutions of \eqref{main-eq0}.}
 The main result of this section can then be stated as follows.

\begin{theorem}
\label{tail-thm}
Suppose that $\Phi$ is a bounded positive solution of \eqref{main-eq1}. Then  there are $M^\pm$ such that
$$
\limsup_{\xi\to\infty}\frac{\Phi(\xi)}{e^{\mu_ -\xi}}\le M^+
$$
and
$$
\limsup_{\xi\to -\infty}\frac{\Phi(\xi)}{e^{\mu_+ \xi}}\le M^-,
$$
where $\mu_\pm=\mu_\pm(0)$ and $\mu_\pm(\lambda)$ is as in \eqref{mu-plus-minus-eq}.
\end{theorem}

To prove the above theorem, we first prove a lemma.

For given $M>0$, $R^+>L+L_0$, $R^-<-(L+L_0)$, and $\tau>0$, consider the following nonlocal Dirichlet problems,
\begin{equation}
\label{Nonlocal-D-const-eq1}
\begin{cases}
c v'(\xi)+\int_{\RR} k(\xi-\eta) v(\eta)d\eta-v(\xi)-qv(\xi)=0,\quad R^+< \xi< R^++\tau \cr
v(\xi)=M, \,\, \xi \notin (R^+,R^++\tau),
\end{cases}
\end{equation}
and
\begin{equation}
\label{Nonlocal-D-const-eq2}
\begin{cases}
c v'(\xi)+\int_{\RR} k(\xi-\eta) v(\eta)d\eta-v(\xi)-qv(\xi)=0,\quad R^--\tau< \xi< R^-\cr
v(\xi)=M, \,\, \xi \notin (R^- - \tau,R^-).
\end{cases}
\end{equation}
 Let
\begin{equation}
\label{Nonlocal-D-const-sol-eq1}
\psi_{\tau}^+(\xi)=k_1^+e^{\mu_- (\xi-R^+)}+k_2^+e^{\mu_+ (\xi-R^+)},
 \end{equation}
 and
 \begin{equation}
\label{Nonlocal-D-const-sol-eq2}
\psi_{\tau}^-(\xi)=k_1^-e^{\mu_- (\xi-R^-)}+k_2^-e^{\mu_+ (\xi-R^-)},
 \end{equation}
 where $k_1^+=M\frac{e^{\mu_+ \tau}-1}{e^{\mu_+ \tau}-e^{\mu_- \tau}}$, $k_2^+=M\frac{1-e^{\mu_- \tau}}{e^{\mu_+ \tau}-e^{\mu_- \tau}}$,
 and  $k_1^-=M\frac{1-e^{-\mu_+ \tau}}{e^{-\mu_- \tau}-{e^{-\mu_+ \tau}}}$, $k_2^-=M\frac{e^{-\mu_- \tau}-1}{e^{-\mu_- \tau}-e^{-\mu_+ \tau}}$.

\begin{lemma}
\label{end-behavior-lm1}
 $\psi_{\tau}^+(\xi)$ is a super-solution of  \eqref{Nonlocal-D-const-eq1} and  $\psi_{\tau}^-(\xi)$ is a super-solution of  \eqref{Nonlocal-D-const-eq2},
 { that is, they are super-solutions of \eqref{Nolocal-Deq} with $a=R^+$,  $b=R^++\tau$, and $g(\xi)\equiv M$, and
 $a=R^--\tau$, $b=R^-$, and $g(\xi)\equiv M$, respectively.}
\end{lemma}
\begin{proof}
Consider
\begin{equation}
\label{Tail-eq}
c v'(\xi)+\int_{\RR} k(\xi-\eta) v(\eta)d\eta-v(\xi)-qv(\xi)=0.
\end{equation}
Let $v(\xi)=e^{\mu \xi}$ and then the characteristic equation of \eqref{Tail-eq} becomes
\begin{equation}
\label{Tail-eigen}
c \mu +\int_{\RR} e^{\mu(\eta-\xi)}k(\xi-\eta)d\eta-1-q=g(\mu;0)=0,
\end{equation}
where
$$
g(\mu;0)=c\mu+\int_{\RR} e^{\mu\eta}k(\eta)d\eta -1-q.
$$

 Note that $g(\mu_{\pm};0)=0$. Thus, $\psi(\xi)=k_1e^{\mu_- \xi}+k_2e^{\mu_+ \xi}$ is a solution of \eqref{Tail-eq} { for any choice of the scalars $k_1$ and $k_2$.}

Let $\psi_{\tau}^+$ be as in \eqref{Nonlocal-D-const-sol-eq1}. Note that $\psi_{\tau}^+(R^+)=\psi_{\tau}^+(R^++\tau)=M$. Then there exists a $\theta \in (R^+, R^++\tau)$ such that $(\psi_{\tau}^+)^{'}(\theta)=0$. $(\psi_{\tau}^+)^{''}(\xi)>0$ implies that $(\psi_{\tau}^+)^{'}(\xi)< 0$ for $\xi<\theta$ and $(\psi_{\tau}^+)^{'}(\xi)> 0$ for $\xi>\theta$. Therefore $\psi_{\tau}^+(\xi)\geq M $ for $\xi\leq R^+$ and $\xi \geq R^++\tau$. By definition, $\psi_{\tau}^+(\xi)$ is a super-solution of  \eqref{Nonlocal-D-const-eq1}.

Similarly, we can prove that $\psi_{\tau}^-(\xi)$ is a super-solution of  \eqref{Nonlocal-D-const-eq2}.
\end{proof}

We now prove Theorem \ref{tail-thm}.

\begin{proof} [Proof of Theorem \ref{tail-thm}]
Choose $M=\ds\max_{\xi}{\Phi(\xi)}$. { Observe that $v(t,\xi)=\Phi(\xi)$ is a sub-solution of \eqref{Nolocal-Deq} with
$a=R^+$, $b=R^++\tau$, and $g(\xi)\equiv M$ for any $R^+>L+L_0$ and any $\tau>0$.}
Then { by Lemma \ref{end-behavior-lm1} and}  Proposition \ref{comparison_Nolocal_D-prop}, for any given $R^+>L+L_0$ and $\tau>0$, $\Phi (\xi)\leq \psi_{\tau}^+(\xi)$ for $\xi\in (R^+,R^++\tau)$.
Note that
$$\lim_{\tau\to \infty}\psi_\tau^+(\xi) = M e^{\mu_{-}(\xi-R^+)}.
$$
We then have that
\begin{equation}
\label{tail-eq1}
\Phi (\xi)\leq Me^{\mu_-(\xi-R^+)}\quad \forall\,\, \xi>R^+,
\end{equation}
and thus
$$
0\le \liminf_{\xi\to\infty}\frac{\Phi(\xi)}{e^{\mu_-\xi}}\le \limsup_{\xi\to\infty}\frac{\Phi(\xi)}{e^{\mu_-\xi}}\le M^+:=Me^{-\mu_- R^+} .
$$

Similarly, we have that
\begin{equation}
\label{tail-eq2}
\Phi (\xi)\leq Me^{\mu_+(\xi-R^-)}\quad \forall\,\, \xi<R^-,
\end{equation}
and thus
$$
0\le  \liminf_{\xi\to -\infty}\frac{\Phi(\xi)}{e^{\mu_+\xi}}\le \limsup_{\xi\to -\infty}\frac{\Phi(\xi)}{e^{\mu_+\xi}}\le M^-:=Me^{-\mu_+ R^-} .
$$
\end{proof}

\begin{remark}
\label{tail-rk}
{In general, it remains an open question} whether the limits $\ds\lim_{\xi\to\infty}\frac{\Phi(\xi)}{e^{\mu_-\xi}}$ and $\ds\lim_{\xi\to -\infty}\frac{\Phi(\xi)}{e^{\mu_+ \xi}}$ exist.
\end{remark}

\section{Equivalence of the persistence  and the existence of traveling wave solutions}

In this section, we show that the occurrence of persistence and the existence of traveling wave solutions are equivalent.

\begin{theorem}
\label{equivalence-thm}
The following two statements are equivalent:
\begin{itemize}

\item[(1)] Persistence occurs in \eqref{main-eq0}

\item[(2)] There are traveling wave solutions of \eqref{main-eq0} with speed $c$.
\end{itemize}
\end{theorem}

\begin{proof}
{ First, we prove that (2) implies (1).

 Suppose that (2) holds and that $\Phi$ is a bounded positive solution of \eqref{main-eq1}.
Then $\tilde v(t,\xi;\gamma\Phi):=\gamma \Phi(\xi)$ satisfies
\begin{equation}
\label{new-aux-eq1}
v_t=cv_\xi+\int_{\RR} k(\eta-\xi)v(t,\eta)d\eta-v(t,\xi)+f(\xi,\Phi(\xi))v(t,\xi),\quad \xi\in\RR
\end{equation}
for any $\gamma\in\RR$. Let $v(t,\xi;\gamma\Phi)$ be the solution of \eqref{main-eq} with
$v(0,\xi;\gamma\Phi)=\gamma\Phi(\xi)$. Fix $T>0$. Note that
$\ds\inf_{|\xi|\le L+L_0}\Phi(\xi)>0$ and $v(t,\xi;0)\equiv 0$ for all $t\ge 0$ and $\xi\in\RR$.  Then there is $\gamma_0>0$ such that
for any $0<\gamma<\gamma_0$,
$$
v(t,\xi;\gamma\Phi)\le \Phi(\xi) \quad {\rm for}\quad 0\le t\le T,\,\,  |\xi|\le L+L_0.
$$
 This implies that for any $0<\gamma<\gamma_0$,
$$
f(\xi,\Phi(\xi))\le f(\xi, v(t,\xi;\gamma\Phi)) \quad {\rm for}\quad 0\le t\le T,\,\, \xi\in\RR.
$$
 {Thus  $v(t,\xi)=v(t,\xi;\gamma\Phi)$ is a super-solution of \eqref{new-aux-eq1} because
\begin{align*}
& v_t(t,\xi)-[cv_\xi+\int_{\RR} k(\eta-\xi)v(t,\eta)d\eta-v(t,\xi)+f(\xi,\Phi(\xi)) v(t,\xi)]\\
&=v_t-[cv_\xi+\int_{\RR} k(\eta-\xi)v(t,\eta)d\eta- v(t,\xi)+f(\xi,v)v(t,\xi)]+(f(\xi,v)-f(\xi,\Phi(\xi)))v(t,\xi)\\
&=(f(\xi,v)-f(\xi,\Phi(\xi)))v(t,\xi)\\
&\ge 0.
\end{align*}
Note that $v(t,\xi;\gamma\Phi)- \tilde v(t,\xi;\gamma\Phi) \ge -\beta_0:= -\ds\max_\xi\{\gamma \Phi(\xi)\}$ and then by the comparison principle (see  Proposition \ref{comparison-linear-prop}(1))  applied for \eqref{new-aux-eq1}, } we have that for any $0<\gamma<\gamma_0$,
$$
v(t,\xi;\gamma\Phi)\ge \tilde v(t,\xi;\gamma\Phi)=\gamma \Phi(\xi),\quad \forall\, 0\le t\le T,\,\, \xi\in\RR
$$
and then by the strong comparison principle (see Proposition \ref{comparison-linear-prop}(2))
\begin{equation}
\label{new-aux-ineq1}
v(t,\xi;\gamma\Phi)> \gamma\Phi(\xi), \quad \forall\, t > 0,\,\, \xi\in\RR.
\end{equation}
Note that for any given $u_0\in X^+$ with $\ds\inf_{\xi\in\RR}u_0(\xi)>0$, there is $0<\gamma<\gamma_0$ such that
$$
u_0(\xi)\ge \gamma \Phi(\xi),\quad \forall\,   \xi\in\RR.
$$
Then by Proposition \ref{comparison-linear-prop}(2), we have that
$$
v(t,\xi;u_0)\ge v(t,\xi;\gamma \Phi)\ge \gamma \Phi(\xi),\quad \forall\, t\ge 0,\,\, \xi\in\RR.
$$
This implies that
$$
\liminf_{t\to\infty}\inf_{|\xi|\le K} v(t,\xi;u_0)>0
$$
for any $K>0$. Therefore, (1) holds.
}

 Next, we  prove  that (1) implies (2).

Assume (1) holds, that is, persistence occurs in \eqref{main-eq0}.  {Choose  $M\gg 1$ such that  $u_M(\xi)\equiv M$ is a super-solution of \eqref{main-eq}.   Then {by the definition of persistence in \eqref{main-eq0}}}
$$
\liminf_{t\to\infty} \inf_{|\xi|\le K} v(t,\xi;u_M)>0
$$
for any $K>0$.
By Proposition \ref{comparison-linear-prop},
\begin{equation}
\label{existence-eq00}
v(t,\xi;u_M)\le u_M(\equiv M), \quad \forall\,\, t\ge 0,\,\, \xi \in\RR,
\end{equation}
and {for any $t_2>t_1\ge 0$,
\begin{equation}
\label{existence-eq0}
v(t_2,\xi;u_M)=v(t_1,\xi; v(t_2-t_1,\cdot;M))\le v(t_1,\xi;u_M)\le M,\quad \forall\,\, \xi\in\RR.
\end{equation}}
It then follows that there is $\Phi(\xi)$ such that
$$
\lim_{t\to\infty} v(t,\xi;u_M)=\Phi(\xi),\quad \forall\,\, \xi\in\RR
$$
and
$$
\inf_{|\xi|\le K}\Phi(\xi)>0,\quad \forall\, K>0.
$$

{ In the following, we show that $\Phi(\xi)$ is $C^1$ and $v(t,\xi)=\Phi(\xi)$
is a positive stationary solution of \eqref{main-eq}. To this end},
let $u(t,x;u_M)$ be the solution of \eqref{main-eq0} with $u(0,x;u_M)=u_M$.
{Note that with $\xi=x-ct$,
$$
v(t,\xi;u_M)=u(t,x;u_M).
$$
It then suffices to show that $u_t(t,x;u_M)$ and $u_{tt}(t,x;u_M)$ are uniformly bounded on $\RR^+\times\RR$, and $u_x(t,x;u_M)$ exists and is uniformly bounded and continuous  in $t\ge 0$ and $x\in\RR$.
}

{By (H2), the right hand side of \eqref{main-eq0} with $u(t,x)$ being replaced by $u(t,x;u_M)$ is uniformly continuous in $t\ge 0$ and $x\in\RR$. Hence
$u_t(t,x;u_M)$ is  uniformly continuous in $t\ge 0$ and $x\in\RR$. This implies that the right hand side of \eqref{main-eq0} with $u(t,x)$ being replaced by $u(t,x;u_M)$ is differentiable in $t$ and its derivative with respect to $t$ is  uniformly continuous in $t\ge 0$ and $x\in\RR$.
 It then follows that  $u_{tt}(t,x;u_M)$ exists and is uniformly bounded on $\RR^+\times\RR$.}

{ To show that $u_x(t,x;u_M)$ exists and is uniformly bounded and continuous  in $t\ge 0$ and $x\in\RR$, let}
\begin{equation}
\label{existence-eq1}
w(t,x)=u_t(t,x;u_M)+u(t,x;u_M)-f(x-ct,u(t,x;u_M))u(t,x;u_M).
\end{equation}
By {\eqref{main-eq0}, we have that}
$$
w(t,x)=\int_{\RR}k(y-x)u(t,y;u_M)dy.
$$
This { together with (H1)} implies that $w(t,x)$ is differentiable in $x$ and $w_x(t,x)$ is uniformly bounded and uniformly continuous.
For fixed $x$, let $\phi(t;x)=u(t,x;u_M)$.
By \eqref{existence-eq1}, $\phi(t;x)$ is the solution of
\begin{equation}
\label{exisetnce-eq2}
\begin{cases}
\frac{d\phi}{dt}=-\phi+f(x-ct,\phi)\phi+w(t,x)\cr
\phi(0;x)=w(0,x).
\end{cases}
\end{equation}
{View $x$ as a parameter in the initial value problem \eqref{exisetnce-eq2}. Note that $-\phi+f(x-ct,\phi)\phi+w(t,x)$ and $w(0,x)$ are continuously differentiable in $x\in\RR$. Then by the smooth dependence of solutions of ODEs on the parameters,} $\phi(t;x)$ is differentiable in $x$ (hence $u_x(t,x;u_M)$ exists), and $\psi(t;x)=\phi_x(t;x)$ satisfies
\begin{equation}
\label{existence-eq3}
\begin{cases}
{\frac{d\psi}{dt}=\big(-1+f_u(x-ct,\phi)\phi(t;x)+f(x-ct,\phi) \big)\psi+w_x(t,x)+f_x(x-ct,\phi)\phi(t;x)}\cr
\psi(0;x)=0.
\end{cases}
\end{equation}
Let {$\mu(t;x)=-1+f_u(x-ct,\phi)\phi(t;x)+f(x-ct,\phi)$. By the variation of constants formula}, we have that
\begin{equation}
\label{existence-eq4}
\phi_x(t,x)=\psi(t;x)=\int_0^t e^{\int_\tau^t \mu(s;x)ds}\big(w_x(\tau,x)+f_x(x-c\tau,\phi)\phi(\tau;x)\big)d\tau.
\end{equation}
It is not difficult to see from \eqref{existence-eq4} that $\psi(t;x)$ is uniformly bounded and continuous in $t\ge 0$ and $x\in\RR$.
We then have  that $u_x(t,x;u_M)$ exists and is uniformly bounded and continuous  in $t\ge 0$ and $x\in\RR$.

Recall that
$$
v(t,\xi;u_M)=u(t,x;u_M).
$$
We then have that $v_t(t,\xi;u_M)$, $v_{tt}(t,\xi;u_M)$,  and $v_\xi(t,\xi;u_M)$ exist and are uniformly bounded  and continuous in $t\ge 0$ and $\xi\in\RR$.
This implies that $\Phi^{'}(\xi)$ exists, and
$$
\lim_{t\to\infty} v(t,\xi;u_M)=\Phi(\xi),
$$
$$
\lim_{t\to\infty} v_t(t,\xi;u_M)=0,
$$
and
$$
\lim_{t\to\infty} v_\xi(t,\xi;u_M)=\Phi^{'}(\xi)
$$
locally uniformly in $\xi\in\RR$. It then follows that $v(t,\xi)=\Phi(\xi)$
is a positive stationary solution of \eqref{main-eq}.
This proves that (1) implies (2).
\end{proof}

\begin{remark}
\label{equivalence-rk}
By Theorem \ref{equivalence-thm}, the following two statements are equivalent:
\begin{itemize}
\item[(1)] Extinction occurs in \eqref{main-eq0}

\item[(2)] There are no traveling wave solutions of \eqref{main-eq0} with speed $c$.
\end{itemize}
\end{remark}

\section{Existence, {uniqueness}, and nonexistence of traveling wave solutions}

In this section, we study the existence, {uniqueness}, and nonexistence of positive traveling wave solutions of \eqref{main-eq0} with speed $c$, {or equivalently}, of positive stationary solutions of \eqref{main-eq}. {In the next result we refer to the spreading speed $c^*$ which was defined in formula \eqref{Spreading_c} in section 2.4.}

\begin{theorem}[Existence/nonexistence of traveling wave solutions]
\label{existence-thm}
{ Fix $r,q,L_0>0$.}
\begin{itemize}
\item[(1)] For given $0\le c<c^*$, there is $L^{*}\ge 0$ such that for $L>L^{*}$, \eqref{main-eq} has a positive stationary solution
$v(t,\xi)=\Phi(\xi)$. For $0<L<L^{*}$, there is no positive stationary solution of \eqref{main-eq}.

\item[(2)] For given $c>c^*$, for any $L>0$, there is no positive stationary solution of \eqref{main-eq}.
\end{itemize}
\end{theorem}

\begin{proof}
{ First of all, we note that $f(x,u)$ depends on $r$, $q$, $L$, and $L_0$. For clarity, for fixed $r$, $q$, and $L_0$, we write
$f(x,u)$ as $f(x,u;L)$. By (H2), for $0<L_1<L_2$, we have that
\begin{equation}
\label{xx-eq1}
f(x,u;L_1)\le f(x,u,L_2),\quad {\rm for}\quad 0\le u\le 1,\,\, x\in\RR.
\end{equation}

Next, let $\bar u(x)\equiv 1$. We have that $v(t,x)\equiv 1$ is a super-solution of \eqref{main-eq}.}  Let $v(t,x;\bar u,L)$ be the solution of \eqref{main-eq} with $v(0,x;\bar u,L)=\bar u(x)\equiv 1$.
By the arguments in the proof of Theorem \ref{equivalence-thm},
 there is $\Phi(x;L)$ such that
$$
\lim_{t\to\infty} v(t,x;\bar u,L)=\Phi(x;L)\le 1
$$
locally uniformly in $x\in\RR$, and $v(x)=\Phi(x;L)$
is a {nonnegative} stationary solution of \eqref{main-eq}.
Moreover, {note that $u\equiv 0$ is a solution of \eqref{main-eq}. Then by Proposition \ref{comparison-linear-prop},}  we have that either $\Phi(x;L)\equiv 0$, or $\Phi(x;L)$ is a positive stationary solution of \eqref{main-eq}.
{By \eqref{xx-eq1} and Proposition \ref{comparison-linear-prop}, for any $0<L_1<L_2$, we have that
$$
v(t,x;\bar u,L_1)\le v(t,x;\bar u,L_2),\quad \forall\,\, t\ge 0,\,\, x\in\RR.
$$
Hence}
$$
\Phi(x;L_1)\le \Phi(x;L_2),\quad \forall\,\, x\in\RR,\,\, L_1\le L_2.
$$
Therefore, there is $0\le L^*\le\infty$ such that
$$
\begin{cases}
\Phi(x;L)\equiv 0,\quad \forall \,\, 0<L\le L^*\cr
\Phi(x;L)>0,\quad \forall\,\, L>L^*.
\end{cases}
$$

{ We claim that any nonnegative stationary solution $v^*(x)$ of \eqref{main-eq} satisfies
that
$$
v^*(x)\le 1,\quad \forall\,\, x\in\RR.
$$
In fact, suppose that $v^*(x)$ is a nonnegative stationary solution of \eqref{main-eq}. Let $M=\ds\sup_{x\in\RR}v^*(x)$.
If $M>1$, then $v(t,x)\equiv M$ is a super-solution of \eqref{main-eq}. By the comparison principle
(Proposition \ref{comparison-linear-prop}),
$$
v^*(x)\le v(t,x;M)<M,\quad \forall\,\, t>0,\,\, x\in\RR.
$$
This together with  the tail property (Theorem \ref{tail-thm})  implies that
$$
M=\sup_{x\in\RR}v^*(x)=\max_{x\in\RR} v^*(x)<M,
$$
which is a contradiction. Hence $M\le 1$. It then follows that
$$
v^*(x)\le v(t,x;\bar u,L),\quad \forall\,\, t>0,\,\, x\in\RR
$$
and then
$$
v^*(x)\le \Phi(x;L).
$$ }
Hence, if
$\Phi(x;L)\equiv 0$, then  \eqref{main-eq} has no positive stationary solution.

We now prove (1) and (2).

(1) It suffices to prove that $L^*<\infty$.
 To this end, first, for $0<c<c^*$, choose $c^{'}\in (c,c^*)$ and fix it.  For given $u_0\in X^+$ with nonempty and compact support ${\rm supp}(u_0)$, by Proposition
\ref{c-star-prop},
$$
{\liminf_{t\to\infty}\inf_{|x|\le c^{'}t} (u_\infty(t,x;u_0)-1)=0,}
$$
where $u_\infty(t,x;u_0)$ is the solution of \eqref{main-fixed-domain-eq} with  $u_\infty(0,x;u_0)=u_0(x)$.
 Then we have that
$$
\liminf_{t\to\infty}\inf_{-(c^{'}+c)t\le x\le (c^{'}-c)t} (u_\infty(t,x+ct;u_0)-1)=0.
$$

Next, choose $u_0$ to be $C^1$. {View $x$ as a parameter and by the smooth dependance of solutions of ODEs on the parameters, then} $u_\infty(t,x;u_0)$ is also $C^1$ in $x$ and $v_\infty(t,x;u_0):=u_\infty(t,x+ct;u_0)$ is the solution  of
\eqref{fixed-domain-eq1}
with $v_\infty(0,x;u_0)=u_0(x)$ and satisfies
$$
{\liminf_{t\to\infty}\inf_{-(c^{'}+c)t\le x\le (c^{'}-c)t} (v_\infty(t,x;u_0)-1)=0.}
$$

Now choose $u_0$ such that $u_0\le 1/2$. Then there is $T>0$ such that
$$
v_\infty(T,x;u_0)>1/2,\quad x\in {\rm supp}(u_0).
$$

Let $v(t,x;u_0,L)$ be the solution of \eqref{main-eq} with $v(0,x;u_0,L)=u_0(x)$.
{ By Proposition \ref{basic-convergence},}
$$
\lim_{L\to\infty} v(t,x;u_0,L)= v_\infty(t,x;u_0)
$$
locally uniformly in $(t,x)\in \RR^+\times\RR$. Then
$$
v(T,x;u_0,L)\ge u_0(x), \quad {\rm for}\quad L\gg 1.
$$
Then by {the comparison principle (Proposition \ref{comparison-linear-prop})}, we have that for $L\gg 1$,
{
$$
v(mT,x;u_0,L)\ge v((m-1)T,x;u_0,L)\ge u_0(x),\quad \forall \,\, x\in\RR,\,\, m=1,2,\cdots.
$$
Since $v(mT,x;\bar u,L) \ge v(mT,x;u_0,L)$ and $\ds \Phi(x;L)=\lim_{m \to \infty}v(mT,x;\bar u,L)$ , we have }that
$$
\Phi(x;L)\ge u_0(x),\quad \forall\,\, x\in\RR
$$
for $L\gg 1$. Hence $u(t,x)=\Phi(x;L)$ is a positive stationary solution of \eqref{main-eq} when $L\gg 1$,
and  then $L^*<\infty$.

(2) It suffices to prove that $L^{*}=\infty$. To this end, let $\mu^*>0$ be such that
$$
c^*=\frac{\int_{\RR} e^{\mu^*\eta}k(\eta)d\eta-1+r}{\mu^*}=\inf_{\mu>0} \frac{\int_{\RR} e^{\mu\eta}k(\eta)d\eta-1+r}{\mu}.
$$
Recall that $\mu_-<0<\mu_+$ are such that
\begin{equation}
\label{mu_c}
c \mu_\pm +\int_{\RR} e^{\mu_\pm\eta}k(\eta)d\eta-1-q=0.
\end{equation}
We claim that $\mu^*<|\mu_-|$.

Indeed,
let
$$
h(\mu)=-c\mu+\int_{\RR}e^{\mu \eta}k(\eta)d\eta-1-q.
$$
Then {$h(-\mu_-)=0$}, $h^{'}(\mu)>0$ for $\mu\ge -\mu_-$. Let
$$
\tilde h(\mu)=\frac{\int_{\RR} e^{\mu\eta}k(\eta)d\eta-1-q}{\mu}
$$
for $\mu>0$. 
Note that {$\int_{\RR} \eta^m k(\eta)d\eta=0$ for odd integers $m=1,3,5,...$ and we have that
\begin{align*}
\tilde h^{'}(\mu)&=\frac{\int_{\RR}\eta e^{\mu \eta}k(\eta)d\eta \cdot\mu -(\int_{\RR}e^{\mu\eta}k(\eta)d\eta-1-q)}{\mu^2}\\
&=\frac{1}{\mu^2}\Big[ \int_{\RR} \big(\eta^2 k(\eta)\mu^2+\frac{\eta^4}{3!}k(\eta)\mu^4+\frac{\eta^6}{5!}k(\eta)\mu^6+\cdots\big)d\eta\\
&\quad -\int_{\RR}\big(k(\eta)+\frac{\eta^2}{2!}k(\eta)\mu^2+\frac{\eta^4}{4!}k(\eta)\mu^4+\cdots\big)d\eta+1+q\Big]\\
&=\frac{1}{\mu^2}\Big[ q+\int_{\RR} \big((1-\frac{1}{2!})\eta^2 k(\eta)\mu^2+(\frac{1}{3!}-\frac{1}{4!})\eta^4k(\eta)\mu^4+(\frac{1}{5!}-\frac{1}{6!})\eta^6k(\eta)\mu^6+\cdots\big)d\eta\Big]\\
&>0, \quad {\rm for}\quad \mu>0.
\end{align*}
Then $\mu=-\mu_-=|\mu_-|>0$ is the only solution of $\tilde h(\mu)=c$ in the interval $(0,\infty)$, and
\begin{align*}
\tilde h(\mu^*)&=\frac{\int_{\RR} e^{\mu^*\eta}k(\eta)d\eta-1-q}{\mu^*}\\
&<\frac{\int_{\RR} e^{\mu^*\eta}k(\eta)d\eta-1+r}{\mu^*}=c^*\\
&<c=\frac{\int_{\RR} e^{\mu_-\eta}k(\eta)d\eta-1-q}{-\mu_-}\quad  \text{(by Equation \eqref{mu_c})}\\
&=\frac{\int_{\RR} e^{-|\mu_-|\eta}k(\eta)d\eta-1-q}{|\mu_-|}\\
&=\frac{\int_{\RR} e^{|\mu_-|\eta}k(\eta)d\eta-1-q}{|\mu_-|} \quad \text{(by $k(\eta)=k(-\eta)$)}\\
&=\tilde h(|\mu_-|).
\end{align*}}
It then follows that $|\mu_-|>\mu^*$.

We now assume that \eqref{main-eq} has a positive stationary solution $u=\Phi(\xi)$. By Theorem \ref{tail-thm}, there is $M^+>0$ such that
$$
\Phi(\xi)\le M^+ e^{-|\mu_-|\xi},\quad {\rm for}\quad \xi\gg 1.
$$
Choose $\tilde c\in (c^*,c)$. By Proposition \ref{c-star-prop}, \eqref{main-fixed-domain-eq} has a traveling wave solution
$u(t,x)=\phi(x-\tilde ct)$ {such that
\begin{equation}
\label{phi_fixed_domain}
\phi(-\infty)=1, \phi(+\infty)=0, {\rm and} \lim_{x\to\infty}\frac{\phi(x)}{e^{-\tilde \mu x}}=1,
\end{equation}
}
where $0<\tilde\mu<\mu^*(<|\mu_-|)$ is such that
$$
\tilde c=\frac{\int_{\RR} e^{-\tilde \mu\eta}k(\eta)d\eta-1+r}{\tilde \mu}.
$$
 This implies that
$v(t,x)=\phi(x-(\tilde c-c)t)$ is a { super-solution of \eqref{main-eq}.  Then by (H2), $v(t,x;\gamma)=\gamma \phi(x-(\tilde c-c)t)$ is a  super-solution of \eqref{main-eq} for any $\gamma\ge 1$.  By Theorem \ref{tail-thm}, there is $\gamma\ge 1$ such that
that
$$
\Phi(x)\le \gamma \phi(x).
$$
Hence, by the comparison principle (Proposition \ref{comparison-linear-prop}),}  we have that
$$
\Phi(x)\le \gamma \phi(x-(\tilde c-c)t),\quad \forall\,\, x\in\RR,\,\, t\ge 0.
$$
{Letting} $t\to\infty$, {since $\tilde c<c$ we have that $x-(\tilde c-c)t \to \infty$ and  so by \eqref{phi_fixed_domain}, $\ds \lim_{t \to \infty}\phi(x-(\tilde c-c)t)=\phi(\infty)=0$ , implying that }
$$
\Phi(x)\le 0,\,\, \forall\,\, x\in\RR,
$$
which is a contradiction.

Hence for any $L>0$, \eqref{main-eq} has no positive stationary solution and then $L^{*}=\infty$.
\end{proof}

The following corollary follows directly from the proof of the above theorem.

\begin{corollary}
\label{extinction-cor}
Suppose that there is no positive traveling wave solution of \eqref{main-eq0}. Then for any $u_0\ge 0$,
$$
\lim_{t\to\infty} v(t,x;u_0)=0
$$
locally uniformly in $x\in\RR$, where $v(t,x;u_0)$ is the solution of \eqref{main-eq} with $v(0,x;u_0)=u_0(x)$.
\end{corollary}

{ We will prove the following theorem about the uniqueness of traveling waves of \eqref{main-eq0} by modifying the proof of Theorem 2.1 in \cite{WangZhao}, where the authors dealt with the uniqueness of forced waves (traveling waves) for nonlocal equation \eqref{main-eq0} with nonlinearity \eqref{Reaction-Li}. Due to the different nonlinearity and tail properties of our traveling waves, their proof can not be applied directly.
\begin{theorem}[Uniqueness]
\label{uniqueness-thm}
There are at most one positive bounded solution to Equation \eqref{main-eq1}.
\end{theorem}
\begin{proof}
Let $\Phi_i(\xi),i=1,2$ be two positive bounded solutions of Equation \eqref{main-eq0}, that is,
$\Phi_i$ satisfy that
$$
c \Phi_i'(\xi)+\int_{\RR} \kappa(\eta-\xi) \Phi_i(\eta)d\eta-\Phi_i(\xi)+f(\xi,\Phi_i) \Phi_i(\xi)=0,\quad \xi \in\RR, i=1,2.
$$
Define $\Sigma_{\epsilon}=\{\sigma \ge 1|\sigma\Phi_2 \ge \Phi_1-\epsilon\}$ for $\epsilon>0$. Note that $\Phi_1(\pm \infty)=0$ and then $\Phi_1-\epsilon$ is nonnegative only on a bounded region. Therefore there exists a large enough $\sigma$ such that $\sigma\Phi_2 \ge \Phi_1-\epsilon$, that is, $\Sigma_{\epsilon}$ is not empty. Let $\sigma_\epsilon= \inf \Sigma_\epsilon \ge 1$. Note that  $\sigma_{\epsilon}$ is non-increasing in $\epsilon$. Hence $\ds\lim_{\epsilon \to 0} \sigma_\epsilon$ exists. Let
$\sigma^*=\ds\lim_{\epsilon \to 0} \sigma_\epsilon$. We claim that $\sigma^{*}=1$. If $\sigma^{*}=1$, then we have that $\Phi_2(\xi) \ge \Phi_1(\xi)$. Repeat the previous process by interchanging $\Phi_1$ and $\Phi_2$, and then we also have that  $\Phi_1(\xi) \ge \Phi_2(\xi)$. Thus $\Phi_1 \equiv \Phi_2$.

Now it suffices to prove the claim that $\sigma^{*}=1$. Suppose to the contrary that $\sigma^*>1$. Then
 $\sigma_{\epsilon_0}>1$ for some $\epsilon_0$.  This implies that $\sigma_{\epsilon}>1$ for all $0< \epsilon \le \epsilon_0$. Let $w_\epsilon(\xi)=\sigma_\epsilon\Phi_2(\xi) - \Phi_1(\xi)+\epsilon$. Then, by the definition of $\sigma_{\epsilon}$, $w_\epsilon(\xi) \ge 0$.  There are two cases to show:
\medskip

\noindent {\bf Case 1.} There is $0<\epsilon\le \epsilon_0$ such that  $w_\epsilon (\xi) > 0$ for all $\xi$.
With $\Phi_i(\pm \infty)=0$ for $i=1,2$ we have  that $w_\epsilon (\pm \infty)=\epsilon$ and so $\ds\lim_{\xi \to \pm \infty} {w_\epsilon (\xi)}/{\Phi_2(\xi)}= \infty$. Then in this case ${w_\epsilon(\xi)}/{\Phi_2(\xi)}$ obtains a minimum at some finite $\bar \xi$. Let $\bar \sigma = {w_\epsilon (\bar \xi)}/
{\Phi_2(\bar \xi)}>0$. Then we have that  ${w_\epsilon (\xi)}/{\Phi_2(\xi)} \ge \bar \sigma$, and thus $w_\epsilon (\xi) \ge \bar \sigma \Phi_2$, that is, $(\sigma_\epsilon-\bar \sigma)\Phi_2(\xi) \ge \Phi_1(\xi)-\epsilon$. That causes the contradiction with the definition of $\sigma_\epsilon$ and $\sigma_\epsilon>1$.

\medskip

\noindent {\bf Case 2.} For any $0<\epsilon\le \epsilon_0$, there is $\xi\in \RR$ such that $w_\epsilon(\xi)=0$.
In this case, choose a sequence $0<\epsilon_n \le \epsilon_0$ for $n=1,2... $ such that $\ds\lim_{n \to \infty} \epsilon_n=0$. Let $\xi_n\in\RR$ be such that  $w_{\epsilon_n}(\xi_n)=0$. By $w_{\epsilon_n}(\xi) \ge 0$ for all $\xi\in\RR$,  $w_{\epsilon_n}(\xi)$ obtains the minimum  0 at $\xi_n$. This implies that $w_{\epsilon_n}'(\xi_n)=0$. We claim that $\{\xi_n\}\subset [-(L+L_0),L+L_0]$. Note that $w_{\epsilon_n}(\xi)$ satisfies that
$$
cw_{\epsilon_n}'(\xi)+\int_{\RR} \kappa(\eta-\xi) w(\eta)d\eta-w(\xi)+\sigma_{\epsilon_\epsilon} f(\xi,\Phi_2)\Phi_2(\xi)-f(\xi,\Phi_1)\Phi_1(\xi)=0.
$$
Suppose that  $|\xi_n| >L+L_0$ for some $n$. Then $f(\xi_n,\Phi_2(\xi_n))=f(\xi_n,\Phi_1(\xi_n))=-q$. Plugging $\xi_n$  into the above equation, we have that
\begin{align*}
0&=cw_{\epsilon_n}'(\xi_n)+\int_{\RR} \kappa(\eta-\xi_n) w(\eta)d\eta-w(\xi_n)+\sigma_{\epsilon_n} f(\xi_n,\Phi_2)\Phi_2(\xi_n)-f(\xi_n,\Phi_1)\Phi_1(\xi_n)\\
&=\int_{\RR} \kappa(\eta-\xi_n) w(\eta)d\eta-q(\sigma_{\epsilon_n}\Phi_2(\xi_n)- \Phi_1(\xi_n))\\
&=\int_{\RR} \kappa(\eta-\xi_n) w(\eta)d\eta-q(w(\xi_n)-\epsilon_n)\\
&=\int_{\RR} \kappa(\eta-\xi_n) w(\eta)d\eta+q\epsilon_n\\
&>0,
\end{align*}
which causes a contradiction.
 Hence $\{\xi_n\} \subset [-(L+L_0),L+L_0]$. This implies that there exists a subsequence of $\xi_n$ such that $\ds\lim_{n_k \to \infty} \xi_{n_k}$ exists, denoted by  $\ds\lim_{n_k \to \infty} \xi_{n_k}=\xi^*$. Moreover, as $n_k$ goes to infinity, we have that $w^*(\xi)=\sigma^* \Phi_2(\xi)-\Phi_1(\xi)$ with  $w^*(\xi^*)=0$ and ${w^*}'(\xi^*)=0$. Note that
\begin{align*}
&0=c{w^*}'(\xi)+\int_{\RR} \kappa(\eta-\xi) w^*(\eta)d\eta-w^*(\xi)+\sigma^* f(\xi,\Phi_2)\Phi_2(\xi)-f(\xi,\Phi_1)\Phi_1(\xi)\\
&\ge c{w^*}'(\xi)+\int_{\RR} \kappa(\eta-\xi) w^*(\eta)d\eta-w^*(\xi)-qw^*(\xi).
\end{align*}
Plugging $\xi=\xi^*$ into above inequality, $0 \ge \int_{\RR} \kappa(\eta-\xi^*) w^*(\eta)d\eta$, which implies that $w^*(\xi)=0$ for all $\xi$. Hence, we have that $\Phi_1(\xi)=\sigma^* \Phi_2(\xi)$ and thus
\begin{align*}
0&=c \Phi_1'(\xi)+\int_{\RR} \kappa(\eta-\xi) \Phi_1(\eta)d\eta-\Phi_1(\xi)+f(\xi,\Phi_1) \Phi_1(\xi)\\
&=\sigma^*(c \Phi_2'(\xi)+\int_{\RR} \kappa(\eta-\xi) \Phi_2(\eta)d\eta-\Phi_2(\xi)+f(\xi,\sigma^*\Phi_2) \Phi_2(\xi))\\
&=\sigma^*(c \Phi_2'(\xi)+\int_{\RR} \kappa(\eta-\xi) \Phi_2(\eta)d\eta-\Phi_2(\xi)+f(\xi,\Phi_2) \Phi_2(\xi)) \\
&\quad\quad\quad\quad+(f(\xi,\sigma^*\Phi_2)-f(\xi,\Phi_2)) \sigma^*\Phi_2(\xi)\\
&=(f(\xi,\sigma^*\Phi_2)-f(\xi,\Phi_2)) \sigma^*\Phi_2(\xi),
\end{align*}
which implies that $f(\xi,\sigma^*\Phi_2)-f(\xi,\Phi_2)=0$. In particular, $f(0,\sigma^*\Phi_2(0))-f(0,\Phi_2(0))=0$, which implies that $r(1-\sigma^*\Phi_2(0))=r(1-\Phi_2(0))$ and then $\sigma^*=1$.
\end{proof}}

\section{Spectral theory of nonlocal dispersal operators and its applications}

In this section, we study the spectral theory of the linearized equation of \eqref{main-eq} at $v=0$, i.e.,  \eqref{main-lin-eq1}, and
discuss its applications to the persistence and extinction in \eqref{main-eq}.

Letting $v(t,\xi)=e^{\lambda t} \psi(\xi)$, { \eqref{main-lin-eq1} yields}
\begin{equation}
\label{eigen-eq1}
c\psi'(\xi)+\int_{\RR} k(\eta-\xi) \psi(\eta) d\eta-\psi(\xi) +f(\xi,0)\psi(\xi)=\lambda \psi(\xi).
\end{equation}
It is obvious that $v(t,\xi)=e^{\lambda t} \psi(\xi)$ is a solution of \eqref{main-lin-eq1} if and only if $(\lambda,\psi)$ satisfies Equation \eqref{eigen-eq1}.

Let
$$
X^1=\{u\in X\,|\, u^{'}(\cdot)\in X\}
$$
and
 $$(\mathcal{L}(c) \phi)(\xi):=c\phi'(\xi)+\int_{\RR} k(\eta-\xi) \phi(\eta) d\eta-\phi(\xi) +f(\xi,0)\phi(\xi)$$
for $\phi\in X^1$. Let {$\sigma(\mathcal{L}(c) )$} be the spectrum of $\mathcal{L}(c) $ acting on $X^1$.

\begin{definition}
\label{principal-spectrum-def}
Let
$$
\lambda(c,L) =\sup\{{\rm Re}\lambda\,|\, \lambda\in \sigma(\mathcal{L}(c) )\}.
$$
$\lambda(c,L) $ is called the {\rm principal spectral point} of $\mathcal{L}(c) $ or \eqref{main-lin-eq1}. $\lambda(c,L) $ is called the {\rm principal eigenvalue} of $\mathcal{L}(c) $ if $\mathcal{L}(c) $ has an eigenfunction in $X^+\setminus\{0\}$  associated with $\lambda(c,L)$.
\end{definition}

The objective of this section is to study the properties of
$\lambda(c,L)$ and their applications to the persistence and extinction in \eqref{main-eq}. To do so, we first study in {the} next subsection  the properties of the spectrum {of} \eqref{eigen-eq1} with $f(\xi,0)$ being replaced by some periodic function.

\subsection{Existence of principal eigenvalue with periodic dependence}

In this subsection,
we shall consider the eigenvalue problem \eqref{eigen-eq1}  with   $f(\xi,0)$ being replaced by $a(\xi)$, where $a(\cdot)\in X_p$ and
$$
X_p:=\{a(\cdot) \in X| a(\cdot)=a(\cdot+p)\}
$$
 for $p>0$, that is,
\begin{equation}
\label{eigen-eq}
c\phi'(\xi)+\int_{\RR}k(\eta-\xi) \phi(\eta) d\eta-\phi(\xi) +a(\xi)\phi(\xi)=\lambda \phi(\xi),\quad  \phi\in X_p^1,
\end{equation} where $X_p^1:=\{u \in X |u,u' \in X_p \}$. {We denote $X_p^+=\{u \in X_p| u\ge 0\}$ and $Int(X_p^+)=\{u \in X_p| u > 0\}$, which is the interior of $X_p^+$.}

For given $a\in X_p$, let
 $$(\mathcal{L}(c,a;p) \phi)(\xi):=c\phi'(\xi)+\int_{\RR} k(\eta-\xi) \phi(\eta) d\eta-\phi(\xi) +a(\xi)\phi(\xi)$$
for $\phi\in X^1_p$. Let  $\sigma(\mathcal{L}(c,a;p) )$ be the spectrum of $\mathcal{L}(c,a;p) $ acting on $X_p^1$,
 and
$$
\lambda_p(c,a) =\sup\{{\rm Re}\lambda\,|\, \lambda\in \sigma(\mathcal{L}(c,a;p) )\}.
$$
$\lambda_p(c,a) $ is called the {\it principal spectral point} of $\mathcal{L}(c,a;p) $. $\lambda_p(c,a) $ is called the {\it principal eigenvalue} of $\mathcal{L}(c,a;p) $ if $\mathcal{L}(c,a;p) $ has an eigenfunction in $X_p^+\setminus\{0\}$  associated with $\lambda_p(c,a) $.
We have the following theorem.

\begin{theorem}
\label{pev-thm} Assume that $a$ is Lipschitz continuous. Then
the principal eigenvalue $\lambda_p(c,a) $ of $\mathcal{L}(c,a;p) $ always exists.
\end{theorem}

To prove the above theorem, we first prove some lemmas.
In the following, if no confusion occurs, we may write $\mathcal{L}(c,a;p) $ as $\mathcal{L}$.
 Let $\mathcal{K}:X_p \to X_p$ and $\mathcal{T}: X^1_p \to X_p$ be defined by
$$\mathcal{K}u(\xi)=\int_{\RR} k(\eta-\xi) \phi(\eta) d\eta$$
and
$$(\mathcal{T} \phi)(\xi):=c\phi'(\xi)-\phi(\xi) +a(\xi)\phi(\xi).$$ Then we may write $\mathcal{L}$ as $\mathcal{K}+\mathcal{T}$, and write \eqref{eigen-eq} as
 $$(\mathcal{K}+\mathcal{T})\phi=\lambda \phi.
  $$
  Note that if $(\lambda I-\mathcal{T})^{-1}$ exists, {then \eqref{eigen-eq} has nontrivial solutions $(\lambda,\phi)$ with $\phi$ in $X_p^1\setminus \{0\}$ if and only if}
\begin{equation}
\label{eigen-eq-1}
\mathcal{K} (\lambda I-\mathcal{T})^{-1} v=v
\end{equation}
has nontrivial solutions $(\lambda,v)$ with $v\in X_p\setminus \{0\}$.

For $c=0$, Theorem \ref{pev-thm} has been proved in \cite{ShZh1} (see also \cite{Cov}), that is,
$\lambda_p(0,a)$ is the principal eigenvalue of $\mathcal{L}(0,a;p)$ with a periodic principal eigenfunction. In the rest of the subsection, we assume that $c>0$.

{Let
\begin{equation}
\lambda_{\mathcal{T}}=-1+\overline{a},\quad  \text{where } \bar{a}=\frac{\int_0^{p}a(s)ds}{p}.
\end{equation}}

\begin{lemma}
\label{eigen-prop}
{Assume that $c>0$.}
\begin{itemize}
\item[(1)] If $\alpha \in \CC$ and $Re \alpha > \lambda_{\mathcal{T}}$, then $(\alpha I-\mathcal{T})^{-1}$ exists.

\item[(2)] If $\alpha \in\RR$ and  $\alpha > \lambda_{\mathcal{T}}$,
then $\mathcal{K} (\alpha I-\mathcal{T})^{-1}$ is a compact operator on $X_p$ and is
strongly positive, { i.e.,  $\mathcal{K} (\alpha I-\mathcal{T})^{-1} u \in Int(X_p^+) $ if $u \in X_p^+\setminus\{0\}$}.
\end{itemize}
\end{lemma}

\begin{proof}
(1) For given $w \in X_p$, consider $(\alpha I - \mathcal{T})\phi=w$, i.e.
\begin{equation}
\label{ode-eq1}
\phi'(\xi)-\frac{1}{c}[\alpha +1-a(\xi)]\phi(\xi)=-\frac{w(\xi)}{c}.
\end{equation}
{If a solution $\phi$ in $X_p^1$  exists, then we must have that:
\begin{equation*}
\frac{d}{d\xi}[e^{\frac{1}{c}\int_{\xi}^{0} (\alpha+1-a(\eta))d\eta}\phi(\xi)]=-e^{\frac{1}{c}\int_{\xi}^{0} (\alpha+1-a(\eta))d\eta}\frac{w(\xi)}{c}.
\end{equation*}
Therefore,  we integrate both sides over $[\xi,\infty)$ and exploiting the fact that $\textrm{Re}(\alpha)>-1+{\bar a}$ and that $\phi$ must belong to $X_p^1$, we can
simplify the result to get that
\begin{equation}
\label{new-aux-eq}
\phi(\xi)= \frac{1}{c}\int_{\xi}^{\infty} e^{\frac{1}{c}\int_{\zeta}^{\xi} (\alpha+1-a(\eta))d\eta}w(\zeta)d\zeta.
\end{equation}
For each $\zeta$ in $[\xi, \infty)$, let $k$ be the unique non-negative integer such that $\zeta \in [\xi+kp,\xi+(k+1)p)$. Then
\begin{eqnarray*}
&&|\int_{\zeta}^{\xi} (\bar a-a(\eta))d\eta |= |\int_{\xi}^{\zeta} (\bar a-a(\eta))d\eta | = |\int_{\xi+kp}^{\zeta} (\bar a-a(\eta))d\eta|  \\
&\le& \int_{\xi+kp}^{\zeta} |\bar a-a(\eta)|d\eta \le \int_{0}^{p} |\bar a-a(\eta)|d\eta \le (|\bar a|+ \max_{\xi \in [0,p]}|a(\xi)|)p,
\end{eqnarray*}
and therefore we have that
\begin{align*}
\phi(\xi)&= \frac{1}{c}\int_{\xi}^{\infty} e^{\frac{1}{c}\int_{\zeta}^{\xi} (\alpha+1-a(\eta))d\eta}w(\zeta)d\zeta\\
&=\frac{1}{c}\int_{\xi}^{\infty} e^{\frac{1}{c}\int_{\zeta}^{\xi} (\alpha+1-\bar a)d\eta} e^{\frac{1}{c}\int_{\zeta}^{\xi} (\bar a-a(\eta))d\eta}w(\zeta)d\zeta\\
&\le \Big(e^{\frac{1}{c} \big(|\bar a|+\ds \max_{\xi \in [0,p]}|a(\xi)|\big)p}\|w\|_\infty\Big) \frac{1}{c}\int_{\xi}^{\infty} e^{\frac{1}{c}\int_{\zeta}^{\xi} (\alpha+1-\bar a)d\eta} d\zeta \\
&= \frac{1}{\alpha+1-\bar a} \Big(e^{\frac{1}{c} \big(|\bar a|+\ds \max_{\xi \in [0,p]}|a(\xi)|\big)p}\Big) \|w\|_\infty.
\end{align*}
Then with $C:=\frac{1}{\alpha+1-\bar a} \big(e^{\frac{1}{c} \big(|\bar a|+\ds \max_{\xi \in [0,p]}|a(\xi)|\big)p}\big)$,  we have just shown that $\|\phi\|_\infty \le C \|w\|_\infty$.

Moreover, letting $\hat{\zeta}=\zeta-p$, we have that
\begin{align*}
\phi(\xi+p)&= \frac{1}{c}\int_{\xi+p}^{\infty} e^{\frac{1}{c}\int_{\zeta}^{\xi+p} (\alpha+1-a(\eta))d\eta}w(\zeta)d\zeta\\
&=\frac{1}{c}\int_{\xi}^{\infty} e^{\frac{1}{c}\int_{\hat{\zeta}+p}^{\xi+p} (\alpha+1-a(\eta))d\eta}w(\hat{\zeta}+p)d(\hat{\zeta}+p)\\
&=\frac{1}{c}\int_{\xi}^{\infty} e^{\frac{1}{c}\int_{\hat{\zeta}}^{\xi} (\alpha+1-a(\eta))d\eta}w(\hat{\zeta})d\hat{\zeta}\\
&=\phi(\xi).
\end{align*}

These arguments establish the existence of a solution $\phi$ in $X_p^1$ to equation $(\ref{ode-eq1})$, for each $w$ in $X_p$. Uniqueness of $\phi$ follows from standard results in the theory of linear ODEs. This concludes the proof of part (1) of this Lemma.}


\smallskip

(2) {This} follows from the compactness and the positivity of $\mathcal{K}$ and the strong positivity of  $(\alpha I-\mathcal{T})^{-1}$.
\end{proof}

\begin{lemma}
\label{eigen-prop0} $\quad$
Assume that $c>0$. Then $\lambda_{\mathcal{T}}$ is an eigenvalue of $\mathcal{T}$ and its associated eigenfunction is $\phi(\xi)=e^{\frac{1}{c}\big(\bar{a} \xi-\int_0^{\xi}a(s)ds \big)}$.
\end{lemma}

\begin{proof}
{This can be verified by direct computation.}
\end{proof}

\begin{lemma}
\label{eigen-prop1} $\quad$ Assume that $\alpha> \lambda_{\mathcal{T}}$ and {$c>0$}.
Let $\rho(\alpha)$ be the spectral radius of $\mathcal{K} (\alpha I-\mathcal{T})^{-1}$.
\begin{itemize}
\item[(1)]$\rho(\alpha_1)> \rho(\alpha_2)$ if $ \alpha_2 > \alpha_1$.
\item[(2)]$\rho(\alpha) \to 0$ as $\alpha \to \infty$.
\item[(3)] $\rho(\alpha)$ is continuous in $\alpha> \lambda_{\mathcal{T}}$.
\item[(4)] For some $\ds\alpha_0 >\lambda_{\mathcal{T}}$, if $\rho(\alpha_0)>1$ then there exists a $\lambda$ such that $\rho(\lambda)=1$.
\end{itemize}
\end{lemma}

\begin{proof}
(1) By Lemma \ref{eigen-prop},  $\mathcal{K} (\alpha I-\mathcal{T})^{-1}$ is a { strongly} positive and compact operator and then by Theorem 19.3 (Krein-Rutman theorem) in \cite{Deimling}, $\rho(\alpha)$ is its principal eigenvalue with a positive eigenvector.
Note that for given $w\in X_p$,
$$(\alpha I-\mathcal{T})^{-1} w=\frac{1}{c}\int_{\xi}^{\infty} e^{\frac{1}{c}\int_{\zeta}^{\xi} (\alpha+1-a(\eta))d\eta}w(\zeta)d\zeta.
$$
Hence for { $w\in X_p^+ \setminus \{0\}$},
$$
(\alpha_1 I-\mathcal{T})^{-1} w>(\alpha_2 I-\mathcal{T})^{-1} w
$$
and then
$$
\mathcal{K}(\alpha_1 I-\mathcal{T})^{-1} w>\mathcal{K}(\alpha_2 I-\mathcal{T})^{-1} w.
$$
{By Theorem 19.3 (d) in \cite{Deimling}}, $\rho(\alpha_1)>\rho(\alpha_2)$.

(2) {From the arguments of Lemma \ref{eigen-prop} (1), we have that $\ds\lim_{\alpha \to \infty}\|(\alpha I-\mathcal{T})^{-1} w\|=0$ for any $w\in X_p$. The assertion holds since $\rho(\alpha)\leq \|\mathcal{K}\|\|(\alpha I-\mathcal{T})^{-1}\|$.}

(3) {This} follows from Lemma 2 in \cite{Burger}.

(4) {This} follows from (1), (2), and (3).
\end{proof}

\begin{lemma}
\label{eigen-prop2}
 {Assume that $c>0$.} Let $V(t;c,a)$ be the solution operator of \eqref{main-lin-eq1} with $f(\xi,0)$ being replaced by $a(\xi)$, that is, $v(t,\cdot;v_0)=V(t;c,a)v_0$ is the solution of \eqref{main-lin-eq1} with $f(\xi,0)$ being replaced by $a(\xi)$ and
$v(0,\cdot;v_0)=v_0(\cdot)\in {X_p}$. Then
$$
\lambda_p(c,a)=\limsup_{t\to\infty}\frac{\ln \|V(t;c,a)\|}{t}.
$$
\end{lemma}

\begin{proof}
{This follows from similar arguments as those in} \cite[Proposition 2.5]{HuShVi}. {For completeness,
we provide a sketch of} the proof in the following.
\smallskip

First, let $\ds\lambda_L=\limsup_{t\to\infty}\frac{\ln \|V(t;c,a)\|}{t}$. For any given $\lambda>\ds\limsup_{t\to\infty}\frac{\ln \|V(t;c,a)\|}{t}$,
there is $M>0$ such that
$$
\|V(t;c,a)\|\le M e^{\lambda t} \quad \forall\, \, t\ge 0.
$$
Then for any $\epsilon>0$, we have that
\begin{equation}
\label{new-add-estimate}
\|e^{(-\lambda-\epsilon)t}V(t;c,a)\|\le M e^{-\epsilon t}\quad \forall\,\, t\ge 0.
\end{equation}
Let $\tilde v=e^{(-\lambda-\epsilon)t}v$. Then $\tilde v$ satisfies
\begin{equation}
\label{new-add-eq1}
\tilde v_t=c\tilde v_x+\int_{\RR}k(y-x)\tilde v(t,y)dy-\tilde v(t,x)+a(x)\tilde v(t,x)-(\lambda+\epsilon)\tilde v.
\end{equation}

Let $\tilde V(t;c,a)$ be the solution operator of \eqref{new-add-eq1}. For { any $w\in X^1$, let
\begin{equation}
\label{integral_sol}
\tilde v(t,x)=\int_{-\infty}^t \tilde V(t-\tau;c,a)w(\cdot)d\tau.
\end{equation}
Then, by Leibniz integral rule : for $ -\infty <a(x),b(x)<\infty$,
$$ {\frac {d}{dx}}\left(\int _{a(x)}^{b(x)}f(x,t)\,dt\right)=f{\big (}x,b(x){\big )}\cdot {\frac {d}{dx}}b(x)-f{\big (}x,a(x){\big )}\cdot {\frac {d}{dx}}a(x)+\int _{a(x)}^{b(x)}{\frac {\partial }{\partial x}}f(x,t)\,dt,$$ we have that
\begin{align*}
\frac{\partial \tilde v(t,x)}{\partial t}&=w(x)+\int_{-\infty}^t \frac {\partial }{\partial t} \tilde V(t-\tau;c,a)w(\cdot)d\tau\\
&=c\tilde v_x+\int_{\RR}k(y-x)\tilde v(t,y)dy-\tilde v(t,x)+a(x)\tilde v(t,x)-(\lambda+\epsilon)\tilde v+w(x).
\end{align*}
Thus it is a solution of
\begin{equation}
\label{new-add-eq2}
\tilde v_t=c\tilde v_x+\int_{\RR}k(y-x)\tilde v(t,y)dy-\tilde v(t,x)+a(x)\tilde v(t,x)-(\lambda+\epsilon)\tilde v+w(x)
\end{equation}
on $t\in\RR$. 
Letting $t \to \infty$ in \eqref{integral_sol},
$$
\lim_{t \to \infty}\tilde v(t,x)={ \int_{0}^{\infty}} \tilde V(\tau;c,a)w(\cdot)d\tau=:\tilde v(x;w).
$$

 Moreover,
$$
\|\tilde v(\cdot;w)\|\le \frac{M}{\epsilon}\|w\|.
$$
Suppose that $\tilde v_1(x;w)$ and $\tilde v_2(x;w)$ are two stationary solutions of \eqref{new-add-eq2} and then  $\tilde v_1(x;w)-\tilde v_2(x;w)$ is a stationary solution of Equation \eqref{new-add-eq1}. The estimate \eqref{new-add-estimate} implies that  Equation \eqref{new-add-eq1} has only trivial  stationary solution and so $\tilde v_1(x;w)=\tilde v_2(x;w)$, that is, $\tilde v(x;w)$ is the unique stationary solution of \eqref{new-add-eq2}. This implies that $(\mathcal{K}+\mathcal{T}-(\lambda+\epsilon)I)^{-1}$ exists for any $\epsilon>0$ and so  $(\lambda , \infty)$ is in the resolvent of  the operator $\mathcal{K}+\mathcal{T}$. While $\lambda_p(c,a)=\sup\{{\rm Re}\lambda\,|\, \lambda\in \sigma(\mathcal{K}+\mathcal{T})\}$, we have that $\lambda+\epsilon>\lambda_p(c,a)$.}
Hence, as $\epsilon>0$ was arbitrary,
$$
\lambda_p(c,a)\le \limsup_{t\to\infty}\frac{\ln \|V(t;c,a)\|}{t}.
$$

Next, for any {  $\epsilon>0$ and $M>0$,  let $\bar \lambda=\lambda_p(c,a)+\epsilon$  and  $v_M(x)$ be the unique solution of
\begin{equation}
\label{new-add-eq3}
cv_x+\int_{\RR}k(y-x)v(y)dy-v(x)+a(x)v-\bar \lambda v(x)+M=0,\quad x\in\RR.
\end{equation}

Then with $\tilde{\lambda}=\frac{1+q+|\bar \lambda|}{c}>0$ and  $\tilde{f}(x,v)= \frac{1}{c}(\int_{\RR}k(y-x)v(y)dy+(q+a(x))v(x)+(|\bar  \lambda|-\bar \lambda) v(x)+M) \ge \frac{M}{c}$, Equation \eqref{new-add-eq3} is equivalent to that:
$$
cv_x+\int_{\RR}k(y-x)v(y)dy-v(x)-qv(x)-|\bar \lambda| v(x)+(q+a(x))v+(|\bar \lambda|-\bar \lambda) v(x)+M=0,\quad x\in\RR,
$$
and thus
\begin{equation}
\label{new-add-eq3_0}
\tilde{\lambda} v- v'=\tilde{f}(x,v),\quad x\in\RR.
\end{equation}
Then multiply \eqref{new-add-eq3_0} by $e^{-\tilde{\lambda} x}$ to get that $[-e^{-\tilde{\lambda} x}v(x)]'=e^{-\tilde{\lambda} x}\tilde{f}(x,v(x))$, and integrate both sides over $[x,\infty)$ to
 $$
 v(x)=e^{\tilde{\lambda} x}\int_x^\infty(e^{-\tilde{\lambda} s}\tilde{f}(s,v(s)))ds\ge e^{\tilde{\lambda} x}\int_x^\infty(e^{-\tilde{\lambda} s}\frac{M}{c}) ds=\frac{M}{c\tilde{\lambda}} ,\quad x\in\RR.
$$
Choose $M \ge |\bar \lambda|+1+q$ and then we have that
$v_M(x) \ge 1$.

Note that $v_M(x)$ is a super-solution of \eqref{new-add-eq1}. {By the comparison principle  (Proposition \ref{comparison-linear-prop})} for  \eqref{new-add-eq1}, we have that
$$
0<e^{-\bar \lambda t}V(t;c,a)\cdot 1\le v_M(x),\quad \forall t\ge 0,\,\, x\in\RR.
$$
This implies that
$$
\limsup_{t\to\infty}\frac{\ln\|V(t;c,a)\|}{t}\le\limsup_{t\to\infty}\frac{\ln \|V(t;c,a)\cdot 1\|}{t}\le \bar \lambda,
$$
for all $\epsilon>0$, }and thus also that $\ds\limsup_{t\to\infty}\frac{\ln\|V(t;c,a)\|}{t} \leq \lambda_p(c,a)$.
{This concludes the proof of the lemma.}
\end{proof}

We now prove Theorem  \ref{pev-thm}.

\begin{proof}[Proof of Theorem \ref{pev-thm}]
{For $c=0$, it has been proved in \cite{ShZh1}. Now we assume that $c>0$.}

{Suppose that $\lambda>\lambda_{\mathcal{T}}$ and let}  $(\lambda_{\mathcal{T}}, \phi) $ be as in Lemma \ref{eigen-prop0} such that
$$
c  \phi'(\xi) - \phi(\xi)+a(\xi) \phi(\xi)=\lambda_{\mathcal{T}} \phi(\xi),
$$
and then
$$
-c  \phi'(\xi)+ (\lambda+1-a(\xi)) \phi(\xi)=(\lambda-\lambda_{\mathcal{T}}) \phi(\xi),
$$
denoted by
$$
(-c \p_\xi+(\lambda+1-a(\cdot))I )\phi=(\lambda-\lambda_{\mathcal{T}}) \phi.
$$
Hence
$$
(-c\p_\xi +(\lambda+1-a(\cdot))I)^{-1} (\lambda-\lambda_{\mathcal{T}})\phi=\phi.
$$
This implies that
$$
\mathcal{K}(-c \p_\eta+(\lambda+1-a(\cdot))I)^{-1} (\lambda-\lambda_{\mathcal{T}})\phi=\mathcal{K}\phi>(\lambda-\lambda_{\mathcal{T}})\phi
$$
for $0<\lambda-\lambda_{\mathcal{T}}\ll 1$. It then follows that
$$
\rho(\mathcal{K}(-c\p_\eta+(\lambda+1-a(\cdot))I)^{-1})>1
$$
for $0<\lambda-\lambda_{\mathcal{T}}\ll 1$.

By Lemma \ref{eigen-prop1}
$$
\rho(\mathcal{K}(-c\p_\eta+(\lambda+1-a(\cdot))I)^{-1})\to 0
$$
as $\lambda\to \infty$. Hence there are $\hat{\lambda}>\lambda_{\mathcal{T}}$ and a p-periodic positive function $\psi$  such that
$$
\rho(\mathcal{K}(-c\p_\eta+(\hat{\lambda}+1-a(\cdot))I)^{-1})=1.
$$
\begin{equation}
\label{aaux-eq3}
\int_{-\infty}^\infty \kappa(\eta-\xi) \big( -c\p_\eta +(\hat{\lambda} +1-a(\cdot))I\big)^{-1}\psi(\eta)d\eta=\psi(\eta).
\end{equation}
and thus
\begin{equation}
\label{aaux-eq4}
c\tilde \phi^{'}(\xi)+\int_{-\infty}^\infty \kappa(\eta-\xi)\tilde \phi(\eta)d\eta-\tilde \phi(\xi)+a(\xi)\tilde \phi(\xi)=\hat{\lambda} \tilde \phi(\xi),
\end{equation}
where $\tilde \phi(\xi)=\big( -c\p_\xi +(\hat{\lambda}+1-a(\cdot))I\big)^{-1}\psi(\xi)$. {Note that ${\tilde \phi(\xi)}>0$ on $\RR$ because $\psi(\xi)>0$ on $\RR$, and since $\big( -c\p_\xi +(\hat{\lambda}+1-a(\cdot))I\big)^{-1}$ is strongly positive (see Lemma \ref{eigen-prop}).}Therefore, $\hat{\lambda} \in \sigma(\mathcal{L})$.

Next, we show that $\hat{\lambda}=\lambda_p(c,a)$. Let $V(t;c,a)$ be as in Lemma  \ref{eigen-prop2}, the solution operator of \eqref{main-lin-eq1} with $f(\xi,0)$ being replaced by $a(\xi)$, We have {that}
$$
V(t;c,a)\tilde \phi=e^{\hat{\lambda} t}\tilde\phi\quad \forall\, \, t>0.
$$
{ For any $\phi \in Int(X_p^{+})$ with $\|\phi\|_{\infty}=1$, there exist positive $\sigma_1$ and $\sigma_2$ such that  $\sigma_1 \tilde \phi \le \phi \le \sigma_2 \tilde \phi$.  By comparison principle (Proposition \ref{comparison-linear-prop}),
$\sigma_1 V(t;c,a)\tilde \phi \le V(t;c,a)\phi \le \sigma_2 V(t;c,a)\tilde \phi$ and thus $\sigma_1 \|V(t;c,a)\tilde \phi\| \le \|V(t;c,a)\phi\| \le \sigma_2 \|V(t;c,a)\tilde \phi\|$. Since $\phi$ is arbitrary and $\|\phi\|_{\infty}=1$,  $\sigma_1 \|V(t;c,a)\tilde \phi\| \le \|V(t;c,a)\| \le \sigma_2 \|V(t;c,a)\tilde \phi\|$. Therefore, $\frac{ln(\sigma_1 \|V(t;c,a)\tilde \phi\|)}{t} \le \frac{ln(\|V(t;c,a)\|)}{t} \le \frac{ln(\sigma_2 \|V(t;c,a)\tilde \phi\|)}{t}$ for $t>0$. Then $\frac{ln(\sigma_1 \|e^{\hat{\lambda} t}\tilde \phi\|)}{t} \le \frac{ln(\|V(t;c,a)\|)}{t} \le \frac{ln(\sigma_2 \|e^{\hat{\lambda} t}\tilde \phi\|)}{t}$, and thus  $\hat{\lambda}+\frac{ln(\sigma_1 \|\tilde \phi\|)}{t} \le \frac{ln(\|V(t;c,a)\|)}{t} \le \hat{\lambda}+\frac{ln(\sigma_2 \|\tilde \phi\|)}{t}$. Letting $t \to \infty$, we have that}
$$
\limsup_{t\to\infty}\frac{\ln \|V(t;c,a)\|}{t}=\hat{\lambda}
$$
and hence   $\hat{\lambda}=\lambda_p(c,a)$ by Lemma \ref{eigen-prop2} .
\end{proof}

The following lemma shows the dependence of $\lambda_p(c,a)$ on  $a(\xi)$.

\begin{lemma}
\label{eigen-prop3}
$\lambda_p(c,a_1) \leq \lambda_p(c,a_2)$ whenever $a_1(\xi) \leq a_2(\xi)$. Moreover, $\lambda_p(c,a_1) < \lambda_p(c,a_2)$ if $a_1(\xi) \leq a_2(\xi)$ and $a_1(\xi) \neq a_2(\xi)$.
\end{lemma}

\begin{proof}
{ With Lemma \ref{eigen-prop2} and the comparison principle (Proposition \ref{comparison-linear-prop}), we have that
$$
\lambda_p(c,a_1)=\limsup_{t\to\infty}\frac{\ln \|V(t;c,a_1)\|}{t}\le\limsup_{t\to\infty}\frac{\ln \|V(t;c,a_2)\|}{t}=\lambda_p(c,a_2).
$$
We shall prove the second statement by contradiction. Suppose that $\lambda_p(c,a_1)=\lambda_p(c,a_2):=\bar \lambda_p $ and $\phi_i(\xi)$  is the corresponding positive principal eigenfunction to $\lambda_p(c,a_i)$ for $i=1,2$, that is,
$$
c\phi_i^{'}(\xi)+\int_{-\infty}^\infty \kappa(\eta-\xi)\phi_i(\eta)d\eta-\phi_i(\xi)+a_i(\xi) \phi(\xi)=\bar \lambda_p \phi_i(\xi).
$$
Then
$$
\int_{-\infty}^\infty \kappa(\eta-\xi)\phi_i(\eta)d\eta=-c\phi_i^{'}(\xi)+\phi_i(\xi)-a_i(\xi) \phi(\xi)+\bar \lambda_p \phi_i(\xi),
$$
denoted by $\mathcal{K} \phi_i= (-c\p_\xi+(\bar\lambda_p+1-a_i(\cdot))I)\phi_i$ for $i=1,2$.
Hence we have that $\mathcal{K}(-c\p_\eta+(\bar\lambda_p+1-a_i(\cdot))I)^{-1}w=w$ with $w(\xi)=-c\phi_i^{'}(\xi)+\phi_i(\xi)-a_i(\xi) \phi(\xi)+\bar \lambda_p \phi_i(\xi)=\int_{-\infty}^\infty \kappa(\eta-\xi)\phi_i(\eta)d\eta>0$ for $i=1,2$.
This implies that $\rho(\mathcal{K}(-c\p_\eta+(\bar\lambda_p+1-a_i(\cdot))I)^{-1})=1$ for $i=1,2$. On the other hand, by the arguments in the proof of Lemma \ref{eigen-prop1}(1), we have that  $\rho(\mathcal{K}(-c\p_\eta+(\bar\lambda_p+1-a_1(\cdot))I)^{-1})>\rho(\mathcal{K}(-c\p_\eta+(\bar\lambda_p+1-a_2(\cdot))I)^{-1})$ if $a_1(\xi) \leq a_2(\xi)$ and $a_1(\xi) \neq a_2(\xi)$, which is a contradiction.
}
\end{proof}

\subsection{Dependence of principal eigenvalue on moving speed $c$ and patch size $L$}

In this subsection, we explore some important properties of $\lambda(c,L)$. Recall that $\lambda(c,L)$ is the principal
spectrum point of the spectral problem \eqref{eigen-eq1}, that is, the spectral problem associated to the linearization of \eqref{main-eq} at the trivial solution $v\equiv 0$. In particular, we study the dependence of $\lambda(c,L)$ on $c$ and $L$.

The main results of this subsection are stated in the following theorem.

\begin{theorem}
\label{new-aux-main-thm}

\begin{itemize}

\item[(1)] $\lambda(c,L)$ is a principal eigenvalue. Moreover, let $\phi(\xi)$ be a corresponding
positive eigenfunction,  then
There are $\tilde M_\pm$ such that
$$
\limsup_{\xi\to\infty}\frac{\phi(\xi)}{e^{\mu_-(\lambda(c,L))\xi}}\le \tilde M_+
$$
and
$$
\limsup_{\xi\to -\infty}\frac{\phi(\xi)}{e^{\mu_+(\lambda(c,L))\xi}}\le\tilde M_-,
$$
where $\mu_\pm(\lambda)$ is defined in \eqref{mu-plus-minus-eq}.

\item[(2)] If $0< c< c^*$, there is $0\le L^{**}<\infty$ such that $\lambda(c,L) >0$ for all $L> L^{**}$, and for any $0<L<L^{**}$, $\lambda(c,L) \le 0$.

\item[(3)] If $c>c^*$, then $\lambda(c,L) < 0$.

\end{itemize}
\end{theorem}

 {{To prove the above Theorem we shall first prove some auxiliary results. Pick} $\frac{p}{2}>L_0+L$ and define a p-periodic function $a_p(\xi;L,L_0)$ as follows:}
$$
a_p(\xi;L,L_0)=\begin{cases}
f(\xi,0), \xi \in [-\frac{p}{2},\frac{p}{2}] \cr
f_{p}(\xi+p,0)=f_{p}(\xi,0).
\end{cases}
$$
Observe that
$$
a_p(\xi;L,L_0)\ge a_{2p}(\xi;L,L_0)\ge \cdots\ge a_{2np}(\xi;L,L_0)\ge \cdots { \ge f(\xi,0)\ge -q},
$$
and then { with $a_{2np}(0)=r>-q$, Lemma \ref{eigen-prop3} implies that}
$$
\lambda_p(c,a_p)\ge \lambda_{p}(c,a_{2p})\ge \cdots\ge \lambda_{p}(c,a_{2np})\ge\cdots{ > \lambda_{p}(c,-q)=-q.}
$$
{Then letting $p \to \infty$, the limit of $\lambda_{p}(c,a_{2np})$ exists, and let}
$$
\lambda_\infty(c,L,L_0)=\lim_{n\to\infty}\lambda_{p}(c,a_{2np}).
$$

\begin{proposition}
\label{new-aux-prop2}
 Let $V(t;c,L)$ be the solution operator of \eqref{main-lin-eq1}, that is, $v(t,\cdot;v_0)=V(t;c,L)v_0$ is the solution of \eqref{main-lin-eq1} with
$v(0,\cdot;v_0)=v_0(\cdot)\in X$. Then
$$
\lambda(c,L)=\limsup_{t\to\infty}\frac{\ln \|V(t;c,L)\|}{t}.
$$
\end{proposition}

\begin{proof}
{Note that the periodicity plays no role in the proof of Lemma \ref{eigen-prop2} and we can replace $X_p$ by $X$.}
\end{proof}

{ \begin{remark}
\label{monotone-rk}
Let $a(\cdot)\in X$ and $\lambda(a)$ be the principal spectral point of the eigenvalue problem \eqref{eigen-eq1} on $X$ with $f(\cdot,0)$ being replaced by $a(\cdot)$. Similar to Lemma \ref{eigen-prop3}, we have that
$$
\lambda(a_1)\le \lambda(a_2)
$$
for $a_1,a_2\in X$ with $a_1(x)\le a_2(x)$ ($x\in\RR$). {By observation, 1 is a positive principal eigenvector of $\lambda(-q)$ and  $\lambda(-q)=-q$. Since $f(\xi,0) \ge -q$, we have that}
$$
\lambda(c,L)\ge \lambda(-q)=-q.
$$
\end{remark}
}

\begin{proposition}
\label{new-aux-prop1}
$\lambda(c,L)=\lambda_\infty(c,L,L_0)$ and $\lambda(c,L)$ is an eigenvalue.
\end{proposition}

\begin{proof}
First, let {$\phi_{2np}(\xi)$} be the positive {2np-periodic eigenfunction corresponding to {$\lambda_p(c,a_{2np})$}  with $\|\phi_{2np}\|_\infty=1$, that is, $(\lambda_p(c,a_{2np}),\phi_{2np})$ satisfies that
\begin{equation}
\label{eigen_p_eq}
c\phi_{2np}'(\xi)+\int_{\RR}\kappa(\xi-\eta)\phi_{2np}(\eta)d\eta-\phi_{2np}(\xi)+a_{2np}(\xi;L,L_0)\phi_{2np}(\xi)
=\lambda_p(c,a_{2np})\phi_{2np}(\xi),
\end{equation}
Let
{$\xi_{2np}\in[-np,np]$} be such that  $\phi_{2np}(\xi_{2np})=\ds\sup_{\xi\in\RR}\phi_{2np}(\xi)=1$. Plugging $\phi_{2np}(\xi_{2np})=1$ and $\phi_{2np}'(\xi_{2np})=0$ into Equation \eqref{eigen_p_eq}, we have that
$$\lambda_p(c,a_{2np})=\int_{\RR}\kappa(\xi_{2np}-\eta)\phi_{2np}(\eta)d\eta-1+a_{2np}(\xi_{2np};L,L_0).$$
Note that $\phi_{2np}\le 1$ but not identical to 1 and so $\int_{\RR}\kappa(\xi-\eta)\phi_{2np}(\eta)d\eta<1$ for any $\xi \in \RR$. Then $\lambda_p(c,a_{2np})< a_{2np}(\xi_{2np};L,L_0)$.
On the other hand, with Lemma \ref{eigen-prop3}, we have that $-q < \lambda_p(c,a_{2np})$ and so $-q<a_{2np}(\xi_{2np};L,L_0)$.
This implies that $\xi_{2np}\in (-L-L_0,L+L_0) \subset [-np,np]$.

Next, recall that  $-q< \lambda_p(c,a_{2np})<  a_{2np}(\xi_{2np};L,L_0)$. Then by Equation \eqref{eigen_p_eq}, we have that
\begin{align*}
c|\phi_{2np}^{'}(\xi)|&=|-\int_{\RR}\kappa (\xi-\eta)\phi_{2n p}(\eta)d\eta+\phi_{2np}(\xi){(-\lambda_p(c,a_{2n p})+1-a_{2n p}(\xi;L,L_0))}|\\
&\le\int_{\RR}\kappa (\eta)\|\phi_{2n p}\|_{\infty}d\eta+\|\phi_{2n p}\|_{\infty}(|\lambda_p(c,a_{2n p})|+1+\|a_{2n p}\|_{\infty})\\
&\le2\|\phi_{2n p}\|_{\infty}(1+\|a_{2n p}\|_{\infty}).
\end{align*}
Thus with $\ds\|a_{2np}\|_{\infty}=\max_{\xi\in\RR}|a_{2np}(\xi;L,L_0)|=\max\{r,q\}$ and $\|\phi_{2n p}\|_{\infty}=1$, we have that
$$
 \sup_{\xi\in\RR}|\phi^{'}_{2np}(\xi)|\le\frac{2}{c}\Big(1+\max\{r,q\}\Big).
$$
Therefore} there is $n_k\to\infty$ such that
$\xi_{2n_k p}\to \xi^\infty\in [-L-L_0,L+L_0]$ and $\phi_{2n_kp}(\xi)\to \phi_\infty(\xi)$ locally uniformly in $\xi\in\RR$.
Moreover, we have {that} $\p_\xi\phi_\infty(\xi)$ exists and
$$
c\p_\xi\phi_\infty(\xi)+\int_{\RR}\kappa(\xi-\eta)\phi_\infty(\eta)d\eta-\phi_\infty(\xi)+f(\xi,0)\phi_\infty(\xi)=\lambda_\infty(c,L,L_0)\phi_\infty(\xi).
$$
Since $\phi_\infty(\xi_\infty)=1$, { we have that $\phi_\infty(\xi)\not\equiv 0$. By the arguments as in item (2) of
Proposition \ref{comparison-linear-prop},}
we have that $\phi_\infty(\xi)>0$ for all $\xi\in\RR$.
Therefore
$$
\lambda_\infty(c,L,L_0)\le \lambda(c,L).
$$

Now, since {$f(\xi,0)\le a_{2np}(\xi)$}, {  by Remark \ref{monotone-rk}}, we have that $\lambda(c,L)\le \lambda(c,a_{2np})$ for all $n\ge 1$. This implies that
$$
\lambda(c,L)\le \lambda_\infty(c,L,L_0).
$$
The proposition then follows.
\end{proof}

We now prove Theorem \ref{new-aux-main-thm}.

\begin{proof}[Proof of Theorem \ref{new-aux-main-thm}]

(1) { By Proposition \ref{new-aux-prop1}},  $\lambda(c,L)$ is a principal eigenvalue.  Let $\phi(\xi)$ be a corresponding
positive eigenfunction. {Recall that} $\lambda(c,L)>-q$. By  the similar arguments as in the proof of Theorem \ref{tail-thm},
(1) thus follows.

\smallskip

(2) {By Remark \ref{monotone-rk},  $\lambda(c,L)$ is non-decreasing in $L$. Hence
 there is $0\le L^{**}\le \infty$ such that
$\lambda(c,L)>0$ for $L>L^{**}$ and $\lambda(c,L)\le 0$ for $L<L^{**}$.  It suffices to show that there exists an $L>0$ such that $\lambda(c,L)>0$.}

To this end, first, for $0<c<c^*$, take $c^{'}\in (c,c^*)$ and fix it. Consider \eqref{main-fixed-domain-eq} with $r(1-u)u$ being replaced by $r(1-\epsilon-u)u$ for some $0<\epsilon\ll 1$. For given $u_0\in X^+$ with nonempty and compact support ${\rm supp}(u_0)$, by Proposition
\ref{c-star-prop},
$$
{\liminf_{t\to\infty}\inf_{|x|\le c^{'}t} (u_\infty(t,x;u_0)-(1-\epsilon))=0,}
$$
where $u_\infty(t,x;u_0)$ is the solution of \eqref{main-fixed-domain-eq} with $r(1-u)u$ being replaced by $r(1-\epsilon-u)u$ and $u_\infty(0,x;u_0)=u_0(x)$.
 Then we have that
$$
{\liminf_{t\to\infty}\inf_{-(c^{'}+c)t\le x\le (c^{'}-c)t} (u_\infty(t,x+ct;u_0)-(1-\epsilon))=0.}
$$

Next, {it was proved in Theorem \ref{existence-thm} that  $u_\infty(t,x;u_0)$ is  $C^1$ in $x$ if $u_0 \in X^1$,}  and $v_\infty(t,x;u_0):=u_\infty(t,x+ct;u_0)$ is the solution of
\begin{equation}
\label{fixed-domain-eq2}
v_t=cv_x+\int_{\RR}\kappa(y-x)v(t,y)dy-v(t,x)+r(1-\epsilon-v)v,\quad x\in\RR
\end{equation}
with $v_\infty(0,x;u_0)=u_0(x)$ and satisfies
$$
{\liminf_{t\to\infty}\inf_{-(c^{'}+c)t\le x\le (c^{'}-c)t}(v_\infty(t,x;u_0)-(1-\epsilon))=0.}
$$

Now choose $u_0$ such that {$u_0\le (1-3\epsilon)/2$. Then there is $T_0>0$ such that for any $T>T_0$,
\begin{equation}
\label{u0-veq0}
v_\infty(T,x;u_0)>(1-\epsilon)/2,\quad x\in {\rm supp}(u_0).
\end{equation}

Let $\tilde v_\infty(t,x;u_0)$ be the solution of
\begin{equation}
\label{fixed-domain-eq3}
v_t=cv_{x}+\int_{\RR}\kappa(y-x)v(t,y)dy-v(t,x)+r(1-\epsilon) v,\quad x\in\RR.
\end{equation}
Note that $\tilde v_\infty(t,x;u_0)$ is a solution and thus also a super-solution of \eqref{fixed-domain-eq3}, while  $v_{\infty}(t,x;u_0)$ is a sub-solution of \eqref{fixed-domain-eq3} because
\begin{align*}
&\frac{\partial v_{\infty}}{\partial t}-(c\frac{\partial v_{\infty}}{\partial x}+\int_{\RR}\kappa(y-x)v_{\infty}(t,y;u_0)dy-v_{\infty}(t,x)+r(1-\epsilon) v_{\infty})\\
&=\frac{\partial v_{\infty}}{\partial t}-(c\frac{\partial v_{\infty}}{\partial x}+\int_{\RR}\kappa(y-x)v_{\infty}(t,y;u_0)dy-v_{\infty}(t,x;u_0)+r(1-\epsilon-v_\infty) v_{\infty})\\
&\quad \quad\quad\quad\quad+r(1-\epsilon-v_{\infty}) v_{\infty}-r(1-\epsilon) v_{\infty}\\
&=r(1-\epsilon-v_\infty) v_{\infty}-r(1-\epsilon) v_{\infty}\\
&=-rv_\infty^2\\
&\le 0.
\end{align*}
With \eqref{u0-veq0} and $\tilde v_\infty(0,x;u_0)=v_\infty(0,x;u_0)=u_0 $, by the comparison principle (Proposition \ref{comparison-linear-prop}) for \eqref{fixed-domain-eq3}, we have that
\begin{equation}
\label{u0-veq}
\tilde v_\infty(T,x;u_0)\ge v_\infty(T,x;u_0)\ge (1-\epsilon)/2,\quad \forall x\in {\rm supp}(u_0), T>T_0.
\end{equation}}
Let $v(t,x;u_0,L)$ be the solution of \eqref{main-lin-eq1} with $v(0,x;u_0,L)=u_0(x)$. {Replacing $f(\xi,0)$ in \eqref{main-lin-eq1} by $f(\xi,0)-r\epsilon$, we have that
\begin{equation}
\label{main-lin-eq1_0}
\frac{\partial v(t,\xi)}{\partial t}=c \frac{\partial v(t,\xi)}{\partial \xi}+ \int_{\RR} \kappa(\eta-\xi) v(t,\eta)d\eta-v(t,\xi)+(f(\xi,0)-r\epsilon) v(t,\xi),\quad \xi \in\RR.
\end{equation}
Then \eqref{main-lin-eq1_0} has a solution $e^{-r\epsilon t}v(t,x;u_0,L)$.
{ Apply Proposition \ref{basic-convergence}} with replacing \eqref{main-lin-eq1} by \eqref{main-lin-eq1_0} and \eqref{fixed-domain-eq1} by  \eqref{fixed-domain-eq3} and get that}
$$
\lim_{L\to\infty} e^{-r\epsilon t}v(t,x;u_0,L)=\tilde v_\infty(t,x;u_0)
$$
locally uniformly in $(t,x)\in \RR^+\times\RR$. {Then there exists a large enough $L$ such that
$$
e^{-r\epsilon T}v(T,x;u_0,L) \ge  \tilde v_\infty(T,x;u_0)-\epsilon,\quad \forall x\in {\rm supp}(u_0),T>T_0.
$$
 Thus, with \eqref{u0-veq} and $u_0$ chosen such that $u_0 \le (1-3\epsilon)/2$ in the beginning, we have that
$$
v(T,x;u_0,L)\ge e^{r\epsilon T}(\tilde v_\infty(T,x;u_0)-\epsilon)\ge e^{r\epsilon T}((1-\epsilon)/2-\epsilon)=e^{r\epsilon T}(1-3\epsilon)/2 \ge e^{r\epsilon T} u_0(x),
$$for any $x\in {\rm supp}(u_0)$ and $T>T_0$.}
This together with Proposition \ref{new-aux-prop2}  implies that for $L\gg 1$, $\lambda(c,L)\ge r\epsilon>0$.

(3) { By (1),
 $\lambda(c,L)$ is a principal eigenvalue. Hence} there exist a $\phi >0$ such that
\begin{equation}
\label{eigen-eq3}
c\phi'(\xi)+\int_{\RR}\kappa(\xi-\eta)\phi(\eta)d\eta-\phi(\xi)+f(\xi,0)\phi(\xi)=\lambda(c,L)\phi(\xi).
\end{equation}
{Let $\mu^*>0$ and $c^*$ be as in \eqref{Spreading_c}, i.e., $c^*=\frac{\int_{\RR} e^{-\mu^* z}k(z)dz-1 +r}{\mu^*}$. Let $(\lambda^*,\psi)$ be such that $\psi=e^{\mu^* \xi}$ and $\lambda^*=\int_{\RR} e^{-\mu^* z}k(z)dz-1 +r-\mu^* c$.
Then $(\lambda^*,\psi)$ satisfies that
\begin{equation}
\label{eigen-eq4}
-c\psi'(\xi)+\int_{\RR}\kappa(\xi-\eta)\psi(\eta)d\eta-\psi(\xi)+r\psi(\xi)=\lambda^*\psi(\xi),
\end{equation}
because
\begin{align*}
&-c(e^{\mu^* \xi})'+\int_{\RR}\kappa(\xi-\eta)e^{\mu^* \eta}d\eta-e^{\mu^* \xi}+re^{\mu^* \xi}\\
&=-\mu^*c(e^{\mu^* \xi})+\int_{\RR}\kappa(\xi-\eta)e^{\mu^* \eta}d\eta-e^{\mu^* \xi}+re^{\mu^* \xi}\\
&=(-\mu^*c+\int_{\RR}\kappa(\xi-\eta)e^{-\mu^* (\xi-\eta)}d\eta-1+r)e^{\mu^* \xi}\\
&=\lambda^*\psi(\xi).
\end{align*}

Multiply  \eqref{eigen-eq3}  by $\psi$ and \eqref{eigen-eq4} by $\phi$, integrate both sides of the above equations and subtract, then we have that
\begin{align*}
(\lambda(c,L)-\lambda^*)\int_{\RR}\psi\phi d\xi&=\int_{\RR} [c\phi'(\xi)+\int_{\RR}\kappa(\xi-\eta)\phi(\eta)d\eta-\phi(\xi)+f(\xi,0)\phi(\xi)]\psi(\xi)d\xi \\
&-\int_{\RR}[-c\psi'(\xi)+\int_{\RR}\kappa(\xi-\eta)\psi(\eta)d\eta-\psi(\xi)+r\psi(\xi)]\phi(\xi)d\xi\\
&=\int_{\RR} c[\phi'(\xi)\psi(\xi)+ \psi'(\xi)\phi(\xi)]d\xi\\
&+   \int_{\RR}\kappa(\xi-\eta)\phi(\eta)\psi(\xi)d\eta d\xi- \int_{\RR}\kappa(\xi-\eta)\phi(\xi)\psi(\eta)d\eta d\xi \\
&+\int_{\RR} [f(\xi,0)-r]\phi(\xi)\psi(\xi)d\xi.
\end{align*}
In addition, by the results of (1) in this theorem, we have that  $\phi$ is bounded and
$$
\limsup_{\xi\to\infty}\frac{\phi(\xi)}{e^{\mu_-(\lambda(c,L))\xi}}\le \tilde M_+.
$$
Then
$\ds\lim_{\xi \to -\infty}\phi(\xi)\psi(\xi)=\lim_{\xi \to -\infty}\phi(\xi)e^{u^*\xi}=0$  and  with $\mu^*<-\mu_-(\lambda(c,L))$, we have that $$0 \le \ds\lim_{\xi \to \infty}\phi(\xi)\psi(\xi)=\lim_{\xi \to \infty}\phi(\xi)e^{u^*\xi} \le \lim_{\xi \to \infty}\tilde M_+e^{(\mu^*+\mu_-(\lambda(c,L)))\xi}=0.$$   Therefore we have that $\int_{\RR} c[\phi'(\xi)\psi(\xi)+ \psi'(\xi)\phi(\xi)]d\xi=\int_{\RR} c[\phi(\xi)\psi(\xi)]'d\xi=0$.
By Fubini's theorem, $\int_{\RR}\kappa(\xi-\eta)\phi(\eta)\psi(\xi)d\eta d\xi=\int_{\RR}\kappa(\eta-\xi)\phi(\xi)\psi(\eta)d\eta d\xi$. With $k(z)=k(-z)$, we have that
 $\int_{\RR}\kappa(\xi-\eta)\phi(\eta)\psi(\xi)d\eta d\xi- \int_{\RR}\kappa(\xi-\eta)\phi(\xi)\psi(\eta)d\eta d\xi=0$.
Therefore we have that
$$(\lambda(c,L)-\lambda^*)\int_{\RR}\psi\phi d\xi=\int_{\RR} [f(\xi,0)-r]\phi(\xi)\psi(\xi)d\xi \leq 0,$$
and thus $\lambda(c,L) \leq \lambda^*$. Since $\lambda^*=\int_{\RR} e^{-\mu^* z}k(z)dz-1 +r-\mu^* c=\mu^*(c^*-c)$,
$\lambda^*<0$ if $c>c^*$. This implies that $\lambda(c,L) <0$ if $c>c^*$.}
\end{proof}

\subsection{Applications of principal eigenvalue theory}

In this subsection, we discuss the persistence and extinction in \eqref{main-eq} by applying the principal eigenvalue theory
established in the previous subsection. Our main results of this subsection are stated in the following {Theorem \ref{persistence-extinction}. In the statement of Theorem \ref{persistence-extinction}, we use $\mu_{\pm}(\lambda(c,L))$ which were defined in \eqref{mu-plus-minus-eq}. Let $v(t,\xi;v_0)$ be the solution of \eqref{main-eq} with $v(0,\xi;v_0)=v_0(\xi) \in X^+$.}

\begin{theorem}
\label{persistence-extinction}
\begin{itemize}
\item[(1)] (Persistence) If $\lambda(c,L)>0$, then there is a positive stationary solution of \eqref{main-eq}, and for any $K>0$ and
$v_0\in {Int(X^+)}$ satisfying $\ds\liminf_{\xi\to\infty}\frac{v_0(\xi)}{e^{\mu_-(\lambda(c,L))\xi}}>0$ and $\ds\liminf_{\xi\to -\infty}\frac{v_0(\xi)}{e^{\mu_+(\lambda(c,L))\xi}}>0$,
$$
\liminf_{t\to\infty} \inf_{|\xi|\le K} v(t,\xi;v_0)>0.
$$

\item[(2)] (Extinction) If {$\lambda(c,L) \le 0$}, then for any $v_0\in X^+$,
$$
\lim_{t\to\infty} \sup_{\xi\in\RR} v(t,\xi;v_0)=0.
$$
\end{itemize}
\end{theorem}

\begin{proof}
(1) First, assume $\lambda(c,L)>0$. Let $\phi(\xi)$ be a corresponding positive eigenfunction {with $\|\phi\|_{\infty}=1$. Let $\underline{v}(\xi)=\alpha \phi(\xi)$ for $\alpha>0$. Then we have that
\begin{align*}
&-(c \underline{v}'(\xi)+ \int_{\RR} \kappa(\eta-\xi) \underline{v}(\eta)d\eta-\underline{v}(\xi)+f(\xi,\underline{v})\underline{v}(\xi))\\
&=-\alpha(c \phi'(\xi)+ \int_{\RR} \kappa(\eta-\xi) \phi(\eta)d\eta-\phi(\xi)+f(\xi,\alpha\phi)\phi(\xi))\\
&=-\alpha(c \phi'(\xi)+ \int_{\RR} \kappa(\eta-\xi) \phi(\eta)d\eta-\phi(\xi)+f(\xi,0)\phi(\xi))+\alpha(f(\xi,0)-f(\xi,\alpha\phi))\phi(\xi)\\
&=-\alpha\phi(\xi)(\lambda(c,L)-(f(\xi,0)-f(\xi,\alpha\phi))).
\end{align*}
Note that there is $\alpha_0>0$ such that $f(\xi,0)-f(\xi,\alpha\phi))<\lambda(c,L)$ for $0<\alpha<\alpha_0$, and thus $$-(c \underline{v}'(\xi)+ \int_{\RR} \kappa(\eta-\xi) \underline{v}(\eta)d\eta-\underline{v}(\xi)+f(\xi,\underline{v})\underline{v}(\xi)) \le 0.$$ This implies that} $\underline{v}$ is a sub-solution of \eqref{main-eq} for $0<\alpha\le \alpha_0$.

Next, {choose $M > \max\{1,\alpha_0\}$. Let  $\bar v(\xi) \equiv M$. Note that $f(x,M)<0$ if $M>1$. Thus
$-(c \bar{v}'(\xi)+ \int_{\RR} \kappa(\eta-\xi) \bar{v}(\eta)d\eta-\bar{v}(\xi)+f(\xi,\bar v)\bar{v}(\xi))=-f(\xi,M)M \ge 0.$
Hence $\bar v(\xi)$ is a super-solution of \eqref{main-eq}. Note that $\bar v >\underline{v}$.} Let $v(t,\xi;\underline{v})$ be the solution of \eqref{main-eq} with $v(0,\xi;\underline{v})=\underline{v}(\xi)$. Then {by the comparison principle (Proposition \ref{comparison-linear-prop})}
$$
\bar{v} \ge v(t_2,\xi;\underline{v})=v(t_1,\xi;v(t_2-t_1,\cdot;\underline{v}))\ge  v(t_1,\xi;\underline{v}),\quad \forall\,\, 0<t_1<t_2,\,\, \xi\in\RR
$$
and
{$$
\bar{v} \geq v(t,\xi;\underline{v})= v(t,\xi;\alpha \phi(\xi))\geq \underline{v}=\alpha \phi(\xi)\quad \forall\,\, t>0, \xi\in\RR,\,\, 0<\alpha\le\alpha_0.
$$}
It then follows that $\Phi(\xi)$ is a positive stationary solution of \eqref{main-eq}, where
$$
\Phi(\xi)=\lim_{t\to\infty} v(t,\xi;\underline{v}).
$$

{Suppose that $v_0\in Int(X^+)$ with $\ds\liminf_{\xi\to\infty}\frac{v_0(\xi)}{e^{\mu_-(\lambda(c,L))\xi}}>0$ and $\ds\liminf_{\xi\to -\infty}\frac{v_0(\xi)}{e^{\mu_+(\lambda(c,L))\xi}}>0$. Recall that  there are $\tilde M_\pm$ such that
$$
\limsup_{\xi\to\infty}\frac{\phi(\xi)}{e^{\mu_-(\lambda(c,L))\xi}}\le \tilde M_+
\quad {\rm and}\quad
\limsup_{\xi\to -\infty}\frac{\phi(\xi)}{e^{\mu_+(\lambda(c,L))\xi}}\le\tilde M_-.
$$
Then there are $C>0$ and $0<\alpha_1\le\alpha_0$ such that for all $0<\alpha \le \alpha_1$ and $|\xi|>C$, $v_0(\xi)\ge \underline{v}=\alpha \phi(\xi)$. Let $\alpha_2=\frac{\min_{\xi \in [-C,C]}\{v_0(\xi)\}}{\max_{\xi \in [-C,C]}\{\phi(\xi)\}}$. Then for $0<\alpha<\alpha_2$, we also have that $v_0(\xi) \ge \alpha \phi(\xi)$,  for $|\xi|\le C$.
Thus, for $0<\alpha<\min\{\alpha_1,\alpha_2\}$, we have that}
$$
v_0(\xi)\ge \underline{v}=\alpha \phi(\xi),\quad \forall\,\, \xi\in\RR.
$$
It then follows  by { comparison principle (Proposition \ref{comparison-linear-prop}) that
$$
v(t,\xi;v_0)\ge v(t,\xi;\underline{v})\geq \alpha \phi(\xi),\quad \forall\,\, t>0, \xi\in\RR.
$$}
Thus (1) follows.

\smallskip

(2)  {We start the proof by making a stronger assumption that $\lambda(c,L)<0$.  By Proposition \ref{new-aux-prop2}, $\|V(t;c,a)\| \le e^{\lambda(c,L) t} \to 0$, as $t \to \infty$. Then for any $v_0\in X^+$,}
$$
\lim_{t\to\infty}\|V(t;c,a)v_0\|=0.
$$
Note that
{ for $t>0$, $$
v(t,\xi;v_0)=V(t;c,a)v_0+\int_{0}^{t}V(t-s;c,a)[f(\xi,v(s,\xi;v_0))-f(\xi,0)]{v(s,\xi;v_0)}ds.
$$ Since $f_u(x,u)\leq 0$ in (H2) and  $v(t,\xi;v_0)\geq 0$, this implies that $$\int_{0}^{t}V(t-s;c,a)[f(\xi,v(s,\xi;v_0))-f(\xi,0)]{v(s,\xi;v_0)}ds\leq 0.$$
Therefore we have that}
$$
0\le v(t,\xi;v_0)\le V(t;c,a)v_0, \quad \forall\,\, t\ge 0,\,\, \xi\in\RR.
$$
{By the squeeze theorem, $\ds\lim_{t \to \infty}v(t,\xi;v_0)=0$.

 Now we assume that $\lambda(c,L)=0$. Let $\phi$ be a positive principal eigenfunction associated with 0, that is,
$$
c \phi'(\xi)+ \int_{\RR} \kappa(\eta-\xi) \phi(\eta)d\eta-\phi(\xi)+f(\xi,0)\phi(\xi)=0.
$$
Let $\psi(\xi)=\phi(-\xi)$ and then $\psi(\xi)$ satisfies that
\begin{equation}
\label{Eigen-adjoint}
-c \psi'(\xi)+ \int_{\RR} \kappa(\eta-\xi) \psi(\eta)d\eta-\psi(\xi)+f(\xi,0)\psi(\xi)=0.
\end{equation}

Choose M large enough such that  $\bar v=M$ is a super-solution of Equation \eqref{main-eq} and $0 \le v_0 \le \bar v=M$. Then by the comparison principle  (Proposition \ref{comparison-linear-prop}), $0 \le v(t,\xi;v_0) \le v(t,\xi;\bar v) \le M$.
Then $\ds\lim_{t \to \infty}v(t,\xi;\bar v)$ exists and let $w(\xi)=\ds\lim_{t \to \infty}v(t,\xi;\bar v)$ that satisfies
\begin{equation}
\label{w-equation}
c w'(\xi)+ \int_{\RR} \kappa(\eta-\xi) w(\eta)d\eta-w(\xi)+f(\xi,w)w(\xi)=0.
\end{equation}

By the similar arguments in the proof of item (3) of Theorem \ref{new-aux-main-thm},  multiply \eqref{w-equation} by $\psi$ and \eqref{Eigen-adjoint}  by $w$, integrate both sides of the above equations and subtract, then we have that
\begin{equation}
\label{faux-equation}
 \int_{\RR} (f(\xi,w)-f(\xi,0))w(\xi)\psi(\xi)d \xi=0,
\end{equation}
Note that $w(\xi) \ge 0$, $\psi>0$ and by (H2), $f(\xi,w)\le f(\xi,0)$. From \eqref{faux-equation}, if $w(\xi)>0$, then we must have $f(\xi,w) = f(\xi,0)$, which causes a contradiction. Therefore, we must have that $w(\xi)=0$ and so $\ds\limsup_{t \to \infty}v(t,\xi;v_0)\le \ds\lim_{t\to\infty} v(t,\xi;\bar v)=0$.
}
\end{proof}

\begin{remark}
\label{persistence-rk} In the case that $\lambda(c,L)>0$, {it remains an open question} whether Theorem \ref{persistence-extinction}(1) holds
for any $v_0\in X^+\setminus\{0\}$.
\end{remark}

\noindent\textbf{Acknowledgement.} This research is supported in part from NSF-DMS-1411853 of Patrick De Leenheer.

\end{document}